\documentclass[10pt]{amsart}
\usepackage{latexsym, amsmath,amssymb}
\usepackage[abbrev,backrefs]{amsrefs}
\usepackage{graphicx}
\usepackage{xcolor}
\usepackage{enumitem}
\DeclareMathOperator*{\argmin}{arg\,min}

\newcommand{\subscript}[2]{$#1 _ #2$}

 \numberwithin{equation}{section}

\usepackage{amsmath}

\def\XXint#1#2#3{{\setbox0=\hbox{$#1{#2#3}{%
\int}$ }
\vcenter{\hbox{$#2#3$ }}\kern-.6\wd0}}

\renewcommand{\epsilon}{\varepsilon}
\newtheorem{theorem}{Theorem}

\newtheorem{lemma}[theorem]{Lemma}
\newtheorem{corollary}[theorem]{Corollary}

\newtheorem{proposition}[theorem]{Proposition}
\newtheorem{definition}[theorem]{Definition}
\newtheorem{remark}[theorem]{Remark}
\newcommand{\bth}{\begin{theorem}}
\newcommand{\ble}{\begin{lemma}}
\newcommand{\bcor}{\begin{corr}}

\newcommand{\bdeff}{\begin{deff}}

\newcommand{\bprop}{\begin{proposition}}
\newcommand{\ele}{\end{lemma}}
\newcommand{\ecor}{\end{corr}}
\newcommand{\edeff}{\end{deff}}

\numberwithin{theorem}{section}

\newcommand{\eprop}{\end{proposition}}

\newcommand{\supp}{\text{supp }}
\renewcommand{\Pi}{\varPi}

\renewcommand{\epsilon}{\varepsilon}

\newcommand{\R}{{\mathbb R}}
\newcommand{\Z}{{\mathbb Z}}

\usepackage{esint}

\usepackage{comment}

\begin{document}

\title[$\beta$-dimensional sharp maximal function and its applications]
{$\beta$-dimensional sharp maximal function and its applications}

\author[Y.-W.B. Chen]{You-Wei Benson Chen}
\address[Y.-W.B. Chen]{National Taiwan University, Department of Mathematics, Astronomy Mathematics Building 5F, No. 1, Sec. 4, Roosevelt Rd., Taipei 10617, Taiwan (R.O.C.)}
\email{bensonchen.sc07@nycu.edu.tw}

\author[A. Claros]{Alejandro Claros}
\address[A. Claros]{BCAM - Basque Center for Applied Mathematics and Universidad del País Vasco / Euskal Herriko Unibertsitatea, Bilbao, Spain}
\curraddr{}
\email{\tt aclaros@bcamath.org}


\begin{abstract}
In this paper, we study $\beta$-dimensional sharp maximal operator defined as 
\begin{align*}
\mathcal{M}^{\#} _\beta f(x) := \sup_{Q} \inf_{c \in \mathbb{R}} \chi_{Q}(x) \frac{1}{\ell(Q)^\beta} \int_Q |f-c| \; d \mathcal{H}^{\beta}_\infty,
\end{align*}
where the supremum is taken over all cubes in $\R^d$ with sides pararell to the coordinate axes, $\ell(Q)$ is the length side of $Q$ and $\mathcal{H}^{\beta}_\infty$ is the Hausdorff content. In particular, we prove Fefferman-Stein inequality for $\mathcal{M}^{\#} _\beta f$ by giving a good lambda estimate for $\beta$-dimensional sharp maximal operator in the context of Hausdorff content. Additionally, we prove the Muckenhoupt-Wheeden inequality in this framework by establishing a good lambda inequality of independent interest.

\end{abstract}

\maketitle

\section{Introduction}
Introduced by C. Fefferman and E. M. Stein in \cite{FeffermanStein}, the sharp maximal function is defined by
\begin{align*}
\mathcal{M}^{\#}f (x) := \sup_{Q} \chi_Q (x) \frac{1}{|Q|} \int_Q |f - f_Q| \; dy,
\end{align*}
where $f_Q = \frac{1}{|Q|} \int_Q f \, dy$ denotes the average of $f$ over the cube $Q$. In \cite{FeffermanStein}, the Fefferman-Stein inequality was established to study the intermediate spaces between $\operatorname{BMO}$ and $L^p (\mathbb{R}^d)$. This inequality plays an important role in harmonic analysis. Specifically, there exists a constant $C>0$, depending only on the dimension $d$ and the exponent $p$, such that
\begin{align}\label{strongstein}
 \left( \int_{\mathbb{R}^d} \left(\mathcal{M} f \right)^p \; dx  \right)^{\frac{1}{p}}  \leq C \left( \int_{\mathbb{R}^d} \left(\mathcal{M}^{\#} f  \right)^p\; dx \right)^{\frac{1}{p}} \quad \text{ for } 1<p <\infty.
\end{align}
Here, $\mathcal{M}f$ denotes the Hardy-Littlewood maximal function of $f$, defined as
\begin{align*}
    \mathcal{M}f (x) : = \sup_Q \frac{\chi_{Q} (x)}{ |Q|} \int_Q |f|\; dy.
\end{align*}

In the recent paper \cite{Chen-Spector}, a John-Nirenberg inequality is established within the framework of Hausdorff content. More precisely, let \(\beta \in (0,d]\) and let \(Q_0\) be a cube in \(\mathbb{R}^d\). A function \(u\) is said to belong to \(\operatorname{BMO}^\beta(Q_0)\) if  
\begin{equation}\label{nondyadic_norm}
\|u\|_{\operatorname{BMO}^\beta(Q_0)} 
:= \sup_{Q \subset Q_0} \inf_{c \in \mathbb{R}} \frac{1}{\ell(Q)^\beta} \int_{Q} |u - c| \, d\mathcal{H}^{\beta}_\infty 
< +\infty,
\end{equation}
where the supremum is taken over all subcubes \(Q\) of \(Q_0\). In \cite{Chen-Spector}, it is further shown that any function \(u \in \operatorname{BMO}^\beta(Q_0)\) satisfies a John–Nirenberg-type estimate with respect to the Hausdorff content. Namely, there exist constants \(c_\beta, C_\beta > 0\) such that for every subcube \(Q \subset Q_0\), for all \(t > 0\), and for a suitably chosen \(c_Q \in \mathbb{R}\), one has 
\begin{equation}\label{Martinez-Spector-prime}
\mathcal{H}^\beta_\infty\bigl(\{x \in Q : |u(x) - c_Q| > t\}\bigr) 
\leq C_\beta \,\ell(Q)^\beta \exp(-c_\beta \, t /\|u\|_{\operatorname{BMO}^\beta(Q_0)} ).
\end{equation}
Here, \(\mathcal{H}^\beta_\infty\) denotes the Hausdorff content defined by 
\[
\mathcal{H}^\beta_\infty(E) 
:= \inf \Bigl\{\sum_{i=1}^\infty \omega_{\beta} \,r_i^\beta : E \subset \bigcup_{i=1}^\infty B(x_i, r_i)\Bigr\},
\]
where 
\[
\omega_{\beta} 
:= \frac{\pi^{\beta/2}}{\Gamma\bigl(\tfrac{\beta}{2}+1\bigr)}
\]
is a normalization constant. When $\beta = d$, this result recovers the classical John-Nirenberg theorem given in \cite{JN}. For $\beta \in (0,d)$, this result demonstrates a self-improving property for functions whose bounded mean oscillation is controlled on subspaces of dimension $\beta$, quantified via the Hausdorff content. The motivation to define such a space arose from the consideration of certain, sharp forms of the Sobolev embedding in the critical exponent \cite{Yudovich, Adams1973, Cianchi:2008, FontanaMorpurgo, MS}, as in this regime one often finds not only the well-known exponential integrability estimate but even estimates of this type along lower dimensional sets. A key ingredient in the proof of \eqref{Martinez-Spector-prime} in \cite{Chen-Spector} is the observation that, despite the inherent nonlinearity of the Choquet integral with respect to \(\mathcal{H}^{\beta,Q}_\infty\), the maximal function \(\mathcal{M}_{\mathcal{H}^{\beta,Q}_\infty}\) associated with the dyadic Hausdorff content \(\mathcal{H}^{\beta,Q}_\infty\) (see Section 2 for definition) satisfies the weak type \((1,1)\) bound
\[
\mathcal{H}^{\beta,Q}_{\infty}\Bigl(\bigl\{x \in \mathbb{R}^d : \mathcal{M}_{\mathcal{H}^{\beta,Q}_\infty} f(x) > t\bigr\}\Bigr)
\;\le\;
\frac{C}{t} \int_{\mathbb{R}^d} |f| \,d\mathcal{H}^{\beta,Q}_{\infty}.
\]
Additionally, the definition of $\operatorname{BMO}^\beta$ naturally leads to the introduction of the $\beta$-dimensional sharp maximal operator $\mathcal{M}^{\#} _\beta$ defined as follows.
\begin{definition}
Let $0< \beta \leq d\in \mathbb{N}$. The $\beta$-dimensional sharp maximal operator of a function $f$ in $\mathbb{R}^d$ is defined as 
\begin{align}\label{sharpmaximal}
\mathcal{M}^{\#} _\beta f(x) := \sup_{Q} \inf_{c \in \mathbb{R}} \chi_{Q}(x) \frac{1}{\ell(Q)^\beta} \int_Q |f-c| \; d \mathcal{H}^{\beta}_\infty
\end{align}
(see also \cite[p.~16]{harjulehto2024hausdorff} for the $\beta$-dimensional sharp maximal function defined with respect to balls in Euclidean space).
\end{definition}
It is easy to see that $\mathcal{M}^{\#} _\beta f \lesssim_{d,\beta}  \mathcal{M}_{\mathcal{H}^\beta_\infty} f$, where $\mathcal{M}_{\mathcal{H}^\beta_\infty}$ denotes the $\beta$-dimensional maximal operator introduced in \cite{chen2023capacitary}, defined as 
\begin{align*}
 \mathcal{M}_{\mathcal{H}^\beta_\infty} f(x)   = \sup_{r>0}  \frac{1}{\mathcal{H}^{\beta}_{\infty} (B(x,r))} \int_{B(x,r)} |f| \;\;d\mathcal{H}^{\beta}_{\infty}.
\end{align*}

The first main result of this paper extends (\ref{strongstein}) to the framework of the Choquet integral with respect to Hausdorff content. We establish a Fefferman–Stein inequality for the \(\beta\)-dimensional sharp maximal operator \(\mathcal{M}^{\#}_\beta f\), thereby generalizing the result in (\ref{strongstein}). Our approach is founded on a good-\(\lambda\) inequality for \(\mathcal{M}^{\#}_\beta f\), as presented in Theorem \ref{lemmaforcorollay3}. Deriving Theorem \ref{lemmaforcorollay3} poses challenges due to the nonlinearity of the Choquet integral; however, these difficulties are overcome by employing the packing condition for Hausdorff content introduced in \cite{OV}.

\begin{theorem}\label{Maintheorem3}
    Let $0 <  \beta  \leq d \in \mathbb{N}$, $0<p_0<\infty$ and $f\in L^{1}_{loc}(\mathcal{H}^\beta_\infty)$  such that \begin{equation}\label{newassumptionThm1.2}
    \sup_{0< t \le N} t^{p_0} \mathcal{H}^{\beta}_\infty ( \{ x\in \R^d : \mathcal{M}_{\mathcal{H}^\beta_\infty} f (x)>t\}) <\infty \qquad \text{ for all } N>0.
\end{equation}
    Then there exists a constant $C = C(p,d,\beta)>0$ such that 
   \begin{align}\label{Maincorollay3}
 \left(  \int_{\mathbb{R}^d} \left( \mathcal{M}_{\mathcal{H}^\beta_\infty} f  \right)^p \; d\mathcal{H}^\beta_\infty \right)^{\frac{1}{p}} \le C \left( \int_{\mathbb{R}^d} \left( \mathcal{M}^{\#} _\beta f \right)^p \; d\mathcal{H}^\beta_\infty \right)^{\frac{1}{p}}
   \end{align}
   for all $p_0<p<\infty$, and also
   \begin{align}\label{Maincorollay3-2}
   \left\| \mathcal{M}_{\mathcal{H}^\beta_\infty} f \right\|_{L^{p,\infty}(\mathcal{H}^\beta_\infty)} \le C \left\| \mathcal{M}^{\#} _\beta f  \right\| _{L^{p,\infty}(\mathcal{H}^\beta_\infty)}
   \end{align}
   for all $p_0\le p<\infty$. In here, $L^{p,\infty}( \mathcal{H}^\beta_\infty)$ denotes the vector space of all the functions $g: \mathbb{R}^d \to \mathbb{R}$ satisfying
\begin{align*}
   \Vert g \Vert_{L^{p,\infty} (\mathcal{H}^\beta_\infty)}: = \sup_{\lambda> 0 } \lambda \, \mathcal{H}^\beta_\infty\left( \left\{ x\in \mathbb{R}^d: |g(x)| > \lambda \right\} \right)^{\frac{1}{p}}< \infty.
\end{align*}
\end{theorem}

\begin{remark}
    The condition \eqref{newassumptionThm1.2} is quite general, and it is satisfied, for example, for each function $f$ such that
    \begin{equation*}
        \int_{\R^d} |f|^{p_0} d\mathcal{H}_\infty^\beta<\infty,
    \end{equation*}
    with $p_0\ge 1$, by the weak type $(p,p)$ inequality satisfied by $\mathcal{M}_{\mathcal{H}^\beta_\infty}$ (see \cite{chen2023capacitary}). We emphasize that in the previous result, we do not require the function $f$ to be measurable; the result holds for non-measurable functions. In the case $\beta = d$, we recover \eqref{strongstein} while extending it to the setting of non-measurable functions.
\end{remark}

Let \(0 < \alpha < d\). We recall that the Riesz potential of order \(\alpha\) of a measure \(\mu\) is defined by
\begin{align*}
I_\alpha \mu(x) := \frac{1}{\gamma(\alpha)} \int_{\mathbb{R}^d} \frac{d\mu(y)}{|x-y|^{d-\alpha}},
\end{align*}
where \(\gamma(\alpha) = \pi^{d/2} 2^\alpha \Gamma(\alpha/2) \Gamma((d-\alpha)/2)^{-1}\) is a normalization constant.  We shall denote $I_\alpha f$ to represent $I_\alpha \mu$ where $d\mu(x)=f(x)dx$. It is well-known that B. Muckenhoupt and R. L. Wheeden proved in \cite[Theorem 1]{MR0340523} that for \(1 < p < \infty\),
\begin{align}\label{Wheeden}
    \| I_\alpha \mu \|_{L^p (\mathbb{R}^d)} \leq C \| \mathcal{M}_\alpha \mu \|_{L^p (\mathbb{R}^d)},
\end{align}
where the fractional maximal function \(\mathcal{M}_\alpha \mu \) is defined as
\begin{align*}
    \mathcal{M}_\alpha \mu (x)= \sup_{x\in Q}  \frac{\mu(Q)}{\ell(Q)^{d-\alpha}}.
\end{align*}
As before, we shall denote $M_\alpha f$ to represent $M_\alpha \mu$ where $d\mu(x)=f(x)dx$. The proof of (\ref{Wheeden}) employs the good lambda inequality for the fractional maximal function. Subsequently, Adams \cite{Adams1975} provided an alternative approach to \eqref{Wheeden} under the additional assumption that \(f\in L^r(\mathbb{R}^d)\) for some \(r\geq 1\). A key point in his argument is the following pointwise estimate of independent interest. 
\begin{theorem}[Adams \cite{Adams1975}]
Let $0< \alpha < d \in \mathbb{N}$. The pointwise inequality
\begin{equation}\label{Adamspointwise}
\mathcal{M}^\# \bigl(I_\alpha f\bigr)(x) 
\;\cong\; 
\mathcal{M}_\alpha f(x),
\end{equation}
holds for all non-negative measurable functions \(f\) in \(\mathbb{R}^d\) such that \(I_\alpha f \in L^1_{\mathrm{loc}}(\mathbb{R}^d)\) is a tempered distribution. 
\end{theorem}

\begin{remark}
In \cite{Adams1975}, the hypothesis given for \eqref{Adamspointwise} requires \(I_\alpha f \in L^1_{\mathrm{loc}}(\mathbb{R}^d)\) along with the integrability condition
\[
\int_{\mathbb{R}^d} 
\frac{I_\alpha f(x)}{(1+|x|)^{\,d + \alpha}} 
\,dx 
< \infty.
\]
However, a closer examination of \cite[p.~771]{Adams1975} shows that this integrability condition is not strictly necessary. In fact, the weaker assumption that \(I_\alpha f \in L^1_{\mathrm{loc}}(\mathbb{R}^d)\) and is a tempered distribution suffices for the pointwise estimate \eqref{Adamspointwise}.
\end{remark}
Our second main result is a refinement of (\ref{Adamspointwise}) in the context of the \(\beta\)-dimensional sharp maximal operator, formulated as follows:

\begin{theorem}\label{Adamsextend1}
      Let $0< \alpha < d \in \mathbb{N}$, $\beta \in (d-\alpha ,d]$. The pointwise inequality
    \begin{align*}
        \mathcal{M}^{\#}_\beta (I_\alpha f)(x) \cong \mathcal{M}_\alpha f (x)
    \end{align*}
    holds for all non-negative measurable functions \(f\) in \(\mathbb{R}^d\) such that \(I_\alpha f \in L^1_{\mathrm{loc}}(\mathbb{R}^d)\) is a tempered distribution. 
\end{theorem}

Theorem \ref{Adamsextend1} helps us study the local behavior of functions of bounded $\beta$-dimensional mean oscillation, and it provides an analog of \cite[Theorem 1.6]{chen2024selfimproving} concerning the local Morrey space (see Section 2 for precise definition). More precisely, let \(\Omega\) be an open set and \(f \geq 0\) such that \(I_\alpha f \in L^1_{\text{loc}}(\mathbb{R}^d)\). Then we have
\begin{align*}
    \| I_\alpha f \|_{\operatorname{BMO}^\beta(\Omega)} \cong \| f \|_{\mathcal{M}^\alpha (\Omega)} \quad \text{for} \quad \beta \in (d-\alpha, d].
\end{align*}

 Theorem \ref{Maintheorem3} and Theorem \ref{Adamsextend1} naturally led one to inquire whether there is an analog of B. Muckenhoupt and R. L. Wheeden's result in (\ref{Wheeden}) in the context of  \( L^p(\mathcal{H}^{\beta}_\infty) \)-boundedness. We now state the third result of this paper, which contains $(\ref{Wheeden})$ as a particular case with $\beta=d$.

\begin{theorem}\label{wheedenHausdorffcontent}
    Let $0<\alpha<d\in \mathbb{N}$, $\beta \in( d-\alpha, d ]$, $1\le p<\infty$ and $\mu$ be a locally finite measure in $\mathbb{R}^d$. Then there exists $C=C(d,\alpha, \beta)>0$ such that 
    \begin{align}\label{MW strong}
       \left( \int_{\mathbb{R}^d} \left( I_\alpha \mu \right)^{p}\; d \mathcal{H}^{\beta}_\infty\right)^{\frac{1}{p}} \leq C \,  p \left( \int_{\mathbb{R}^d} \left( \mathcal{M}_\alpha \mu \right)^p \; d \mathcal{H}^{\beta}_\infty \right)^{\frac{1}{p}} ,
    \end{align}
    and 
    \begin{align}\label{MW weak}
        \Vert I_\alpha \mu \Vert_{L^{p,\infty} (\mathcal{H}^\beta_\infty)} \leq C \, p \Vert \mathcal{M}_\alpha \mu \Vert_{L^{p,\infty}(\mathcal{H}^{\beta}_\infty)} .
    \end{align}
\end{theorem}
It is noteworthy that the proof of Theorem \ref{wheedenHausdorffcontent} does not rely on Theorem \ref{Adamsextend1}. In fact, if one were to use Theorem \ref{Adamsextend1} to derive Theorem \ref{wheedenHausdorffcontent}, additional assumptions on \(I_\alpha f\) would be required. Instead, our approach employs the good-\(\lambda\) estimate established in Theorem \ref{goodlambdaestimaterelatedRieszandfractional}. Although the nonlinearity of the Choquet integral again introduces some challenges in the proof, these obstacles are overcome by applying the packing condition for Hausdorff content introduced in \cite{OV}.

\begin{remark}
    The linear dependence on $p$ is a consequence of the exponential decay in our good lambda inequality, stated in Theorem \ref{good lambda exp}. 
\end{remark}

\begin{remark}
    The dimensional constraint $\beta\in (d-\alpha, d]$ in Theorems \ref{Adamsextend1} and \ref{wheedenHausdorffcontent} naturally arise from the interaction between the Riesz potential $I_\alpha$ and the Hausdorff content $\mathcal{H}_\infty^\beta$. When $\beta\le d-\alpha$, the convolution kernel $|x|^{\alpha-d}$ fails to be integrable near the origin with respect to the content $\mathcal{H}_\infty^\beta$. This non-integrability implies that $I_\alpha f$ may be infinite in sets of dimension $d-\alpha$ for general functions $f$, as observed by D. Adams \cite[p. 772]{Adams1975}. On the other hand, under additional assumptions as $f\in L^r(\R^d)$ for some $r\in (1,\frac{n}{\alpha})$, then $I_\alpha f$ is in $L^p(\mathcal{H}_\infty^{d-\alpha})$ with $p = \frac{r(n-\alpha )}{n- \alpha r}$, as established in \cite[Theorem 5.9]{harjulehto2024hausdorff}. However, for general functions $f$, the constraint $\beta\in (d-\alpha, d]$ is necessary. 
\end{remark}

The paper is organized as follows: Section 2 reviews the necessary preliminaries for defining the Choquet integral with respect to Hausdorff content and discusses their basic properties, such as Hölder's inequality and packing condition. This section also covers the weak type $(1,1)$ estimate for dyadic maximal function and a minor modification of the Calderón-Zygmund decomposition with respect to the dyadic Hausdorff content presented in \cite[Theorem 3.1]{Chen-Spector}. In Section 3, we give the good lambda estimates for the \(\beta\)-dimensional sharp maximal function \(\mathcal{M}^\#_\beta f\), and we use it to prove Theorem \ref{Maintheorem3}. In section 4, we give the proof of Theorem \ref{Adamsextend1}. In Section 5, we give the good lambda estimates for the fractional maximal function \(\mathcal{M}_\alpha \mu \) and we prove Theorem \ref{wheedenHausdorffcontent}. Finally, in Section 6, we apply Theorem \ref{Thm exp int} to obtain a local exponential integrability estimate for the gradient of weak solutions to a singular $p$-Laplace-type equation with measure data, with respect to Hausdorff content.

\section{Preliminaries}\label{preliminaries}

We begin with recalling the definition of the Choquet integral with respect to an outer measure \( H \). Let \( H \) be an outer measure, that is, \( H \) satisfies the following properties:
\begin{enumerate}[label=(\roman*)]
\item \( H(\emptyset) = 0 \);
\item If \( E \subseteq F \), then \( H(E) \leq H(F) \);
\item If \( E \subseteq \bigcup_{i=1}^\infty E_i \), then 
\begin{align*}
H(E) &\leq \sum_{i=1}^\infty H(E_i).
\end{align*}
\end{enumerate}
The Choquet integral of a non-negative function \( f \) over a set \( \Omega \subseteq \mathbb{R}^d\) with respect to outter measure \( H \) is then defined as:
\begin{align}\label{choquet_def1}
\int_\Omega f \; dH := \int_0^\infty H(\{x \in \Omega : f(x) > t\}) \; dt.
\end{align}

A utilitarian way to deal with the Choquet integral with respect to the Hausdorff content $\mathcal{H}^{\beta}_{\infty}$ is to move between the dyadic Hausdorff content $\mathcal{H}^{\beta,Q}_{\infty}$ which is defined as follows:
\begin{definition}
  Let $0< \beta \le d \in \mathbb{N}$ and $Q$ be a cube in $\mathbb{R}^d$. Then for every $E\subseteq \mathbb{R}^d$, we define the dyadic Hausdorff content $\mathcal{H}^{\beta,Q}_{\infty}$ of $E$ as
\begin{align}\label{DefofHasudorffcontent}
 \mathcal{H}^{\beta,Q}_{\infty}(E):= \inf \left\{\sum_{i=1}^\infty \ell(Q_i)^\beta : E \subset \bigcup_{i=1}^\infty Q_i, \; Q_i \in \mathcal{D}(Q) \right\}
\end{align}
where $\mathcal{D}(Q)$ denotes the collection of all dyadic cube generated by $Q$.
\end{definition}
While the Hausdorff content $\mathcal{H}^{\beta}_{\infty}$ is not strongly subadditive if $\beta <d$, it can be shown that the dyadic Hausdorff content $\mathcal{H}^{\beta, Q}_\infty$ is strongly subadditive for any cube $Q \subseteq \mathbb{R}^d$ and thus satisfies
\begin{align}\label{subadditivitydyadic}
     \int_{\mathbb{R}^d} \sum^\infty_{j=1} f_j \;d \mathcal{H}^{\beta, Q}_\infty \leq \sum^\infty_{j=1}\int_{\mathbb{R}^d} f_j \;d \mathcal{H}^{\beta,Q}_\infty
\end{align}
(see \cite[Proposition 3.5 and 3.6]{STW} for usual dyadic case and \cite[Proposition 2.6 and Proposition 2.10]{Chen-Spector} for general cube $Q$). Moreover, there exists a constant $C_\beta>0$ such that for every cube $Q \subseteq \mathbb{R}^d$,
\begin{align}\label{equivalenceofdyadic}
 \frac{1}{C_\beta} \mathcal{H}^{\beta,Q}_{\infty}(E) \leq  \mathcal{H}^{\beta}_{\infty}(E) \leq C_\beta \mathcal{H}^{\beta,Q}_{\infty}(E)
\end{align}
for all $E \subset \mathbb{R}^d$(see Proposition 2.3 in \cite{YangYuan} and Proposition 2.11 in \cite{Chen-Spector}) and thus
\begin{align}\label{sublinearofHbeta}
    \int_{\mathbb{R}^d} \sum^\infty_{j=1} f_j \;d \mathcal{H}^\beta_\infty \leq C_\beta ^2\sum^\infty_{j=1}\int_{\mathbb{R}^d} f_j \;d \mathcal{H}^\beta_\infty
\end{align}
by definition of Choquet integral (\ref{choquet_def1}).

Next, we outline the basic properties of the Choquet integral with respect to the Hausdorff content. The following lemma can be found in \cite[p.~5]{Petteri_2023}, with its proof referred to in \cite{AdamsChoquet1} and \cite[Chapter 4]{AdamsMorreySpacebook}.

\begin{lemma}\label{basicChoquet}
Let $0<\beta\leq d \in \mathbb{N}$ and $\Omega \subseteq \mathbb{R}^d$. Then the following statements hold:
\begin{enumerate}[label=(\roman*)]
\item For \( a \geq 0 \) and non-negative functions \( f\) in $\mathbb{R}^d$, we have
\begin{align*}
    \int_\Omega a \; f(x)\; d\mathcal{H}^\beta_\infty = a \int_\Omega f(x) \; d\mathcal{H}^\beta_\infty;
\end{align*}
\item For non-negative functions \( f_1 \) and \( f_2 \) in $\mathbb{R}^d$, we have
\begin{align*}
    \int_\Omega f_1(x) + f_2(x)  \; d\mathcal{H}^\beta_\infty \leq 2 \left( \int_\Omega f_1(x) \; d\mathcal{H}^\beta_\infty + \int_\Omega f_2(x) \; d\mathcal{H}^\beta_\infty \right);
\end{align*}
\item Let \( 1 < p<\infty , \) and $p'$ defined by \( \frac{1}{p} + \frac{1}{p'} = 1 \). Then for non-negative functions \( f_1 \) and \( f_2 \) in $\mathbb{R}^d$, we have
\begin{align*}
    \int_\Omega f_1(x) f_2(x) \; d\mathcal{H}^\beta_\infty \leq 2 \left( \int_\Omega f_1(x)^p \; d\mathcal{H}^\beta_\infty \right)^{\frac{1}{p}} \left( \int_\Omega f_2(x)^{p'} \; d\mathcal{H}^\beta_\infty \right)^{\frac{1}{p'}}.
\end{align*}
\end{enumerate}
\end{lemma}

We denote by $\mathcal{M}^\beta(\Omega)$ the vector space of all Radon measure in $\mathbb{R}^d$ satisfying
\begin{align*}
\mathcal{M}^{\beta}(\Omega):= \left\{ \mu \in M_{loc}(\Omega): \|\mu\|_{\mathcal{M}^\beta (\Omega)} := \sup_{x \in \Omega,r>0} \frac{|\mu|(B(x,r) \cap \Omega)}{r^\beta}<\infty \right\},
\end{align*}
where $M_{loc}(\Omega)$ is the set of locally finite Radon measures in $\Omega$.

We now give the definition of the dyadic maximal function $ \mathcal{M}_{\mathcal{H}^{\beta,Q_0}_\infty}$, which is introduced in \cite{Chen-Spector}.

\begin{definition}
Let $0 < \beta \leq d \in \mathbb{N}$ and $Q_0$ be a cube in $\mathbb{R}^d$. The dyadic maximal function associate to $\mathcal{H}^{\beta,Q_0}_\infty$ is defined as 
\begin{align*}
    \mathcal{M}_{\mathcal{H}^{\beta,Q_0}_\infty} f (x) : = \sup_{Q} \frac{1}{\ell (Q)^\beta} \int _{Q} |f| \; d\mathcal{H}^{\beta,Q_0}_\infty,
\end{align*}
where the supremum is taken over all the dyadic cubes $Q$ in $\mathcal{D}(Q_0)$ containing $x$.
\end{definition}


In particular, a weak type $(1,1)$ estimate for dyadic maximal function proved in \cite[Theorem 3.1]{Chen-Spector} asserts that for any $t>0$, the inequality
\begin{align}\label{dyadicweaktypeestimate11}
    \mathcal{H}^{\beta,Q_0}_{\infty}\left(\left\{ x\in \R^d : \mathcal{M}_{\mathcal{H}^{\beta,Q_0}_\infty} f(x)>t\right\}\right) \leq \frac{C}{t} \int_{\mathbb{R}^d} |f|\;d\mathcal{H}^{\beta,Q_0}_{\infty}
\end{align}
holds for all function $f$.
\begin{remark}
Note that in \cite[Theorem 3.1]{Chen-Spector}, $f$ is assumed to be $\mathcal{H}^\beta_\infty$-quasicontinuous (see \cite[p.~2]{Chen-Spector} for definition). However, one finds that the assumption
\begin{align*}
    \int_{\mathbb{R}^d} |f|\; d\mathcal{H}^\beta_\infty < \infty
\end{align*}
is sufficient for the weak type estimate (\ref{dyadicweaktypeestimate11}) by just following the original proof. (That is, the assumption of $f$ being quasicontinuous is not necessary for (\ref{dyadicweaktypeestimate11}).)
\end{remark}

The following lemma, concerning the comparison of the Hausdorff content of level sets of \(\mathcal{M}_{\mathcal{H}^{\beta}_\infty}\) and \(\mathcal{M}_{\mathcal{H}^{\beta,Q_0}_\infty}\), is given in \cite[Lemma 2.4]{chen2023capacitary}.

\begin{lemma}\label{equivoftwomaximal}
    Let $Q_0$ be a cube in $\mathbb{R}^d$. Then there exist constants $C= C(d, \beta)$ and $\Tilde{C} = \Tilde{C}(d,\beta)$ such that
    \begin{align*}
        \mathcal{H}^\beta_\infty\left( \left\{  x \in \mathbb{R}^d : \mathcal{M}_{\mathcal{H}^{\beta}_\infty} f (x) >t \right\} \right) \leq C \mathcal{H}^\beta_\infty\left( \left\{  x\in \mathbb{R}^d : \mathcal{M}_{\mathcal{H}^{\beta,Q_0}_\infty} f(x) > \Tilde{C}   2^{-(\beta +d) } t\right\} \right)
    \end{align*}
   holds for every function $f$.
\end{lemma}

We next record a lemma concerning the packing condition of dyadic cubes, which is a small modification of \cite[Proposition 2.11]{Chen-Spector} and \cite[Lemma 2]{OV}.

\begin{lemma}\label{packing2}
    Suppose $\{Q_j\}$ is a family of non-overlapping dyadic cubes subordinate to some cube lattice $\mathcal{D}(Q)$. Then there exists a subfamily $\{Q_{j_k}\}$ and a family of non-overlapping ancestors $\tilde{Q}_k$ such that
\begin{enumerate}
\item \begin{align*}
\bigcup_{j} Q_j \subset \bigcup_{k} Q_{j_k} \cup \bigcup_{k} \tilde{Q}_k
\end{align*}
\item
\begin{align*}
\sum_{Q_{j_k} \subset Q}\ell(Q_{j_k})^\beta \leq 2\ell (Q)^\beta, \text{ for each dyadic cube } Q.
\end{align*}
\item
 \begin{align*}
      \mathcal{H}^{\beta, Q_0}_\infty(\cup_j Q_j) 
      & \leq   \sum_{k, Q_{j_k} \not \subseteq \tilde{Q}_m} \ell(Q_{j_k})^\beta + \sum_k \ell(\Tilde{Q}_k)^\beta \\
      &\leq  \sum_k \ell(Q_{j_k})^\beta \\
      & \leq  2  \mathcal{H}^{\beta, Q_0}_\infty (\cup_j Q_j).
  \end{align*}
\end{enumerate}

\end{lemma}
\begin{proof}
   An application of Proposition 2.11 in \cite{Chen-Spector} yields a subfamily $\{Q_{j_k} \}_k$ of $\{Q_j \}_j$ and a family of non-overlapping ancestors $\{ \Tilde{Q}_k\}_k$ which satisfy the properties (1), (2) and 

\begin{align*}
\ell(\tilde{Q}_k)^\beta \leq \sum_{Q_{j_i} \subset \tilde{Q}_k}\ell(Q_{j_i})^\beta \text{ for each }\tilde{Q}_k.
\end{align*}
 Moreover, using the subadditive of $\mathcal{H}^{\beta, Q_0}_\infty$, we obtain 
\begin{align}\label{additionalpropertyforpacking2}
     \nonumber \mathcal{H}^{\beta, Q_0}_\infty(\cup_j Q_j) &\leq \sum_{k, Q_{j_k} \not \subseteq \tilde{Q}_m} \ell(Q_{j_k})^\beta + \sum_k \ell(\Tilde{Q}_k)^\beta\\\nonumber
      &\leq \sum_{k, Q_{j_k} \not \subseteq \tilde{Q}_m} \ell(Q_{j_k})^\beta + \sum_{k, Q_{j_k} \subseteq \tilde{Q}_m} \ell(Q_{j_k})^\beta \\ 
      & = \sum_k \ell(Q_{j_k})^\beta .
  \end{align}
Now, it is left to show that 
\begin{align}\label{packingconditionforestimateitselfsunion}
    \sum_k \ell(Q_{j_k})^\beta   \leq  2 \mathcal{H}^{\beta, Q_0}_\infty (\cup_j Q_j).
\end{align}
 To this end, we let $\{Q_m\}_m \subseteq \mathcal{D}(Q_0)$ be a maximal collection of dyadic cubes such that 
\begin{align*}
    \bigcup_j Q_j \subseteq \bigcup_m Q_m.
\end{align*}
The property (2) then yields that
\begin{align*}
    \sum_k \ell(Q_{j_k})^\beta & = \sum_m \sum_{k, Q_{j_k} \subseteq Q_m} \ell(Q_{j_k})^\beta \\
    &\leq \sum_m 2\ell(Q_m)^\beta
\end{align*}
 and (\ref{packingconditionforestimateitselfsunion}) follows by taking the infimum over all such families of dyadic cubes $\{ Q_m\}_m$.
\end{proof}
The condition (2) in Lemma \ref{packing2} is called packing condition, and an important application is that for any $f \geq 0$, one has the estimate
\begin{align}\label{packingestimate}
    \sum_k \int_{Q_{j_k}} f \; d \mathcal{H}^{\beta, Q_0}_\infty \leq 2 \int_{\bigcup_k Q_{j_k}} f \;  \mathcal{H}^{\beta, Q_0}_\infty.
\end{align}

The next result refines the dyadic Calderón–Zygmund decomposition with respect to Hausdorff content, originally obtained in \cite[Theorem 3.4]{Chen-Spector}. Notably, we do \emph{not} require \(f\) to be \(\mathcal{H}^{\beta}_\infty\)-quasicontinuous, a condition assumed in \cite{Chen-Spector}. Our refined version of the theorem will be pivotal in the proof of Theorem~\ref{lemmaforcorollay3}, and we include a complete argument for the reader’s convenience.

\begin{theorem}\label{cz}
Let $Q_0$ be a cube and $\lambda>0$. If \begin{align*}
\mathcal{H}^{\beta, Q_0}_\infty (\{ x \in \mathbb{R}^d: \mathcal{M}_{\mathcal{H}^{\beta,Q_0}_\infty} f(x) > \lambda \}) < \infty,    
\end{align*}
 then there exists a countable collection of non-overlapping dyadic cubes $\{ Q_k\}$ subordinate to $Q_0$ such that 
\begin{enumerate}[label=(\roman*)]
    \item
    $\bigcup_k Q_k = \{ x \in \mathbb{R}^d: \mathcal{M}_{\mathcal{H}^{\beta,Q_0}_\infty} f(x) > \lambda \}$,
     \item $\{Q_k\}$ is the maximal subcollection of all the dyadic cubes contained in $$\{x \in \mathbb{R}^d: \mathcal{M}_{\mathcal{H}^{\beta,Q_0}_\infty} f(x) > \lambda \}.$$
    \item $\lambda< \frac{1}{\ell(Q_k)^\beta} \int_{Q_k} |f| \; d\mathcal{H}^{\beta,Q_0}_\infty \leq 2^\beta \lambda$ for each $k$,
    \item $|f(x)| \leq \lambda$ for $\mathcal{H}^\beta$ -a.e. $x \not\in \bigcup_k Q_k $,
\end{enumerate}
\end{theorem}

Before proving Theorem \ref{cz}, we need two auxiliary results regarding the local behavior of Hausdorff content. The first lemma, stated below, is taken from \cite[Lemma 3.2]{Chen-Spector}.

\begin{lemma}[{\cite[Lemma 3.2]{Chen-Spector}}]\label{lem:ChenSpector3.2}
Let \(Q_0\) be a cube in \(\mathbb{R}^d\), \(E \subset \mathbb{R}^d\)(not necessarily measurable) and \(0 < \beta \le d\). Then 
\[
  \limsup_{Q \to x} \frac{\mathcal{H}^{\beta,Q_0}_\infty (E \cap Q)}{\mathcal{H}^{\beta,Q_0}_\infty (Q)} =1
\quad
\text{for \(\mathcal{H}^\beta\)-almost every \(x \in E\),}
\]
where the limit is taken over all dyadic subcubes \(Q\in \mathcal{D} (Q_0)\) containing \(x\) with \(\ell(Q) \to 0\).
\end{lemma}

Next, we give a result analogous to the classical Lebesgue differentiation Theorem for Hausdorff content without assuming either measurability or quasicontinuity of the function. The proof essentially follows from a close reading of \cite[p.~987]{Chen-Spector}, but we still include it here for the sake of completeness.
\begin{theorem}\label{diffthm}
    Let $0< \beta \leq d \in \mathbb{N}$, $Q_0 \subseteq \mathbb{R}^d$ be a cube and $f: \mathbb{R}^d \to [0,\infty)$ be a function (not necessarily measurable or quasicontinuous). Then for $\mathcal{H}^\beta$-a.e. $x \in \mathbb{R}^d$, we have
\begin{align}
    \limsup_{Q \to x} \frac{1}{\mathcal{H}^{\beta, Q_0}_\infty(Q)} \int_{Q} f\; d\mathcal{H}^{\beta,Q_0}_\infty \geq f(x),
\end{align}
   where the limit on the left-hand side is taken over all dyadic cubes $Q \in \mathcal{D}(Q_0)$ containing $x$ as $\ell(Q) \to 0.$
\end{theorem}

\begin{proof}
    For every $t > 0$, we observe that one has the inclusion
\begin{align}\label{inclusionfotraion}
\nonumber   & \left\{x \in \mathbb{R}^d : \      \limsup_{Q \to x} \frac{1}{\mathcal{H}^{\beta,Q_0}_\infty (Q)} \int_{Q} f\;  d\mathcal{H}^{\beta,Q_0}_\infty <f(x) \right\} \\
    \subset &\bigcup_{t \in \mathbb{Q}^{+}} \left\{x \in \mathbb{R}^d : \      \limsup_{Q \to x} \frac{1}{\mathcal{H}^{\beta,Q_0}_\infty (Q)} \int_{Q} f\;  d\mathcal{H}^{\beta,Q_0}_\infty< t <f(x) \right\}, 
\end{align}
where $\mathbb{Q}^+$ denotes the set of all rational numbers. Moreover, for every $t \in \mathbb{Q}^+$, define
\begin{align*}
    E_t := \{x \in \mathbb{R}^d: f(x) >t \}. 
\end{align*}
   By Lemma~\ref{lem:ChenSpector3.2}, for \(\mathcal{H}^\beta\)-almost every \(x \in E_t\), we have
\begin{align*}
    \limsup_{Q \to x} \frac{1}{\mathcal{H}^{\beta,Q_0}_\infty(Q)} \int_{Q} f\;  d\mathcal{H}^{\beta,Q_0}_\infty
    & \geq  \limsup_{Q \to x} \frac{1}{\mathcal{H}^{\beta,Q_0}_\infty(Q)} \int_{Q \cap E_t} t\;  d\mathcal{H}^{\beta,Q_0}_\infty\\
    & = \limsup_{Q \to x}\frac{\mathcal{H}^{\beta,Q}_\infty(Q \cap E_t)}{\mathcal{H}^{\beta,Q_0}_\infty(Q)}  \ t\\
    & = t.
\end{align*}
 That is, the set 
\begin{align*}
    &E_t \setminus \left\{x \in \mathbb{R}^d : \      \limsup_{Q \to x} \frac{1}{\mathcal{H}^{\beta,Q_0}_\infty (Q)} \int_{Q} f\;  d\mathcal{H}^{\beta,Q_0}_\infty  \geq f(x) \right\}\\
   = & \left\{x \in \mathbb{R}^d : \      \limsup_{Q \to x} \frac{1}{\mathcal{H}^{\beta,Q_0}_\infty (Q)} \int_{Q} f\;  d\mathcal{H}^{\beta,Q_0}_\infty< t <f(x) \right\}
\end{align*}
has $\mathcal{H}^\beta$ measure zero. It then follows from (\ref{inclusionfotraion}) that
\begin{align*}
    &\mathcal{H}^\beta_\infty \left( \left\{x \in \mathbb{R}^d : \      \limsup_{Q \to x} \frac{1}{\mathcal{H}^{\beta,Q_0}_\infty (Q)} \int_{Q} f\;  d\mathcal{H}^{\beta,Q_0}_\infty <f(x) \right\} \right)\\
    \leq & \sum_{t \in \mathbb{Q}^+} \mathcal{H}^\beta_\infty \left(\left\{x \in \mathbb{R}^d : \      \limsup_{Q \to x} \frac{1}{\mathcal{H}^{\beta,Q_0}_\infty (Q)} \int_{Q} f\;  d\mathcal{H}^{\beta,Q_0}_\infty< t <f(x) \right\} \right)\\
    =&  0,
\end{align*}
  which completes the proof.
\end{proof}

We are now ready to give the proof of Theorem  \ref{cz}.

\begin{proof}[Proof of Theorem \ref{cz}]  
    Let $\{ Q_i\}_i$ be the collection of all the dyadic cubes contained in $\{ x \in \mathbb{R}^d: \mathcal{M}_{\mathcal{H}^{\beta,Q_0}_\infty} f(x) > \lambda \}$. Since $\mathcal{H}^{\beta, Q_0}_\infty (\{ x \in \mathbb{R}^d: \mathcal{M}_{\mathcal{H}^{\beta,Q_0}_\infty} f(x) > \lambda \}) < \infty $, the maximal subcollection of $\{ Q_i\}_i$ exists and we denote it as $\{Q_k\}_k$. In particular, we have 
   \begin{align*}
       \bigcup_k Q_k = \{ x \in \mathbb{R}^d: \mathcal{M}_{\mathcal{H}^{\beta,Q_0}_\infty} f(x) > \lambda \}.
   \end{align*}
and thus 
\begin{align*}
    \lambda < \frac{1}{\ell(Q_k)^\beta} \int_{Q_k} |f| \; d \mathcal{H}^{\beta,Q_0}_\infty.
\end{align*}
Now let $\tilde{Q}_k$ denotes the parent of $Q_k$ for each $k$. Using the fact $\tilde{Q}_k$ is not contained in $\{ x \in \mathbb{R}^d: \mathcal{M}_{\mathcal{H}^{\beta,Q_0}_\infty} f(x) > \lambda \}$, we have 
\begin{align*}
    \frac{1}{2^\beta \ell(Q_k)} \int_{Q_k} |f| \; d \mathcal{H}^{\beta,Q_0}_\infty \leq \frac{1}{\ell(\tilde{Q}_k)} \int_{\tilde{Q}_k} |f| \; d \mathcal{H}^{\beta,Q_0}_\infty \leq \lambda.
\end{align*}
Finally, $(iv)$ follows by Theorem \ref{diffthm}. This completes the proof. 
\end{proof}

\section{good lambda estimates for the \(\beta\)-dimensional sharp maximal function and proof of Theorem \ref{Maintheorem3}}

In \cite{FeffermanStein}, Fefferman and Stein proves (\ref{strongstein}) based on the good lambda estimate
\begin{align}\label{Steingoodlambda}
    |\{ x \in \mathbb{R}^d: \mathcal{M}f >t \}| \leq |\{ x \in \mathbb{R}^d: \mathcal{M}^{\#}f > \frac{t}{A}  \}| + \frac{2}{A} | \{ x \in\mathbb{R}^d: 
\mathcal{M}f > 2^{-n-1}t  \}| 
\end{align}
for all positive $t$ and $A$. Here, we give an extension of (\ref{Steingoodlambda}) in the context of Hausdorff content whose proof is parallel to the one given in \cite{FeffermanStein}. There are some difficulties due to the nonlinearity of the Choquet integral, which we overcome through the use of the packing condition in Lemma \ref{packing2}.
\begin{theorem}\label{lemmaforcorollay3}
    Let $0<\beta \leq  d \in\mathbb{N}$, $Q_0$ be a cube in $\mathbb{R}^d$, $f$ be a function in $\mathbb{R}^d$ and $\mu : (0, \infty) \to [0, \infty]$ be defined as
\begin{align}
    \mu(t) : =\mathcal{H}^{\beta, Q_0}_\infty(\{x \in \mathbb{R}^d: \mathcal{M}_{\mathcal{H}^{\beta,Q_0}_\infty} f(x) >t \}).
\end{align}
 Then for any $t>0$ and $A>0$, we have
 \begin{align}
     \mu(t) \leq \mathcal{H}^{\beta, Q_0}_\infty( \{ x \in \mathbb{R}^d: \mathcal{M}^{\#}_{\beta, Q_0} f  (x) > \frac{t}{A}\})+ \frac{8}{A} \mu( 2^{-\beta-2} t) ,
 \end{align}
where $\mathcal{M}^{\#}_{\beta, Q_0}f $ is the dyadic sharp maximal function defined by
\begin{align*}
    \mathcal{M}^{\#}_{\beta, Q_0}f (x) : =  \sup_{ Q \in \mathcal{D}(Q_0)} \inf_{c \in \mathbb{R}} \chi_{Q} (x)\frac{1}{\ell(Q)^\beta} \int_Q |f-c| \; d \mathcal{H}^{\beta, Q_0}_\infty.
\end{align*}
\end{theorem}
\begin{proof}

Fix $t>0$, we apply Theorem \ref{cz} and obtain countable collections of non-overlapping dyadic cubes $\{Q_k^t  \}_k$ and $\{Q_k^{2^{-\beta -2} t}  \}_k$ such that properties (i) to (iv) in  Theorem \ref{cz} hold for $\lambda=t$ and $\lambda = 2^{-\beta-2}t$ respectively.  That is, the collection of cubes $\{Q_k^t  \}_k$ and $\{Q_k^{2^{-\beta -2} t}  \}_k$ satisfy
\begin{enumerate}[label=(\subscript{A}{{\arabic*}})]
    \item
    $\bigcup_k Q_k^t = \{ x \in \mathbb{R}^d: \mathcal{M}_{\mathcal{H}^{\beta,Q_0}_\infty} f(x) > t \}$,
     \item $\{Q_k^t\}_k$ is the maximal subcollection of all the dyadic cubes contained in the set $\{x \in \mathbb{R}^d: \mathcal{M}_{\mathcal{H}^{\beta,Q_0}_\infty} f(x) > t \},$
    \item $t< \frac{1}{\ell(Q_k^t)^\beta} \int_{Q_k^{t}} |f| \; d\mathcal{H}^{\beta,Q_0}_\infty \leq 2^\beta t$,
    \item $|f(x)| \leq t$ for $\mathcal{H}^\beta$ -a.e. $x \not\in \bigcup_k Q_k^t $
\end{enumerate}
and
\begin{enumerate}[label=(\subscript{B}{{\arabic*}})]
    \item
    $\bigcup_k Q_k^{2^{-\beta -2}t } = \{ x \in \mathbb{R}^d: \mathcal{M}_{\mathcal{H}^{\beta,Q_0}_\infty} f(x) > 2^{-\beta -2}t \}$,
     \item $\{Q_k^{2^{-\beta -2} t}\}_k$ is the maximal subcollection of all the dyadic cubes contained in the set $\{x \in \mathbb{R}^d: \mathcal{M}_{\mathcal{H}^{\beta,Q_0}_\infty} f(x) > 2^{-\beta -2}t \},$
    \item $2^{-\beta-2} t < \frac{1}{\ell( Q_k^{2^{-\beta -2}t})^\beta} \int_{Q_k^{2^{-\beta -2}t}} |f| \; d\mathcal{H}^{\beta,Q_0}_\infty \leq \frac{t}{4}$,
    \item $|f(x)| \leq 2^{-\beta -2}t$ for $\mathcal{H}^\beta$ -a.e. $x \not\in \bigcup_k Q_k^{2^{-\beta -2}t} $.
\end{enumerate}
Since $t> 2^{-\beta -2}t$, we have the cubes in $\{ Q_k^t\}_k$ are subcubes of the cubes in $\{ Q_k^{-2^{-\beta-2}t}\}_k$ by ($A_2$) and $(B_2)$. Thus,
\begin{align*}
    \{ Q_k^t\}_k = \mathcal{A}_1 \cup \mathcal{A}_2,
\end{align*}
where 
\begin{align*}
    \mathcal{A}_1 : = \left\{Q_k^t : Q_k^t \subseteq Q_{k_0}^{2^{-\beta-2}t } \text{ for some } k_0 \text{ and }  Q_{k_0}^{2^{-\beta-2}t } \subseteq \{ x \in \mathbb{R}^d : \mathcal{M}^{\#}_{\beta, Q_0}f(x) >\frac{t}{A}  \} \right\},
\end{align*}
and 
 \begin{align*}
      \mathcal{A}_2 : = \left\{Q_k^t : Q_k^t \subseteq Q_{k_0}^{2^{-\beta-2}t } \text{ for some } k_0 \text{ and }  Q_{k_0}^{2^{-\beta-2}t } \not \subseteq \{ x \in \mathbb{R}^d : \mathcal{M}^{\#}_{\beta, Q_0}f(x) >\frac{t}{A}  \} \right\}.
 \end{align*}
 The subadditivity of $\mathcal{H}^{\beta, Q_0}_\infty$ then implies that
\begin{align}
    \mu(t) \leq  \mathcal{H}^{\beta, Q_0}_\infty \left(\bigcup_{Q_k^t \in \mathcal{A}_1} Q_k^t \right) +  \mathcal{H}^{\beta, Q_0}_\infty \left(\bigcup_{Q_k^t \in \mathcal{A}_2} Q_k^t \right).
\end{align}
It is easy to see that 
\begin{align}
    \bigcup_{Q_k^t \in \mathcal{A}_1} Q_k^t \subseteq \left\{x: \mathcal{M}^{\#}_{\beta, Q_0}f(x)>\frac{t}{A} \right\}
\end{align}
and thus 
\begin{align}
    \mathcal{H}^{\beta, Q_0}_\infty \left(\bigcup_{Q_k^t \in \mathcal{A}_1} Q_k^t \right) \leq \mathcal{H}^{\beta, Q_0}_\infty \left(\left\{x: \mathcal{M}^{\#}_{\beta, Q_0}f(x)>\frac{t}{A} \right\} \right).
\end{align}

Now the proof is complete if we show that 
\begin{align}
    \mathcal{H}^{\beta, Q_0}_\infty \left( \bigcup_{Q_k^t \in \mathcal{A}_2} Q_k^t  \right) \leq \frac{8}{A} \mu( 2^{-\beta -2}t ).
\end{align}
To this end, let $\{Q_k' \}_k$ be the subset of $\{ Q_k^{2^{-\beta -2}t}\}_k$ which is defined as 
\begin{align*}
    \left\{Q_k'  \right\}_k: = \left\{ Q \in \{ Q_k^{2^{-\beta -2}t} \}_k :  Q \text{ contains some cube in } \mathcal{A}_2
 \right\}.
\end{align*}
Applying Lemma \ref{packing2} to $\{ Q_k' \}_k$, we then obtain a subfamily $\{ Q_{k_v}'\}_v$ and a family of non-overlapping ancestors $\{  \Tilde{Q}_k \}_k$ such that 
 \begin{align*}
\bigcup_{k} Q_k'  \subset \bigcup_{v} Q_{k_v}' \cup \bigcup_{k} 
\Tilde{Q}_k ,
\end{align*}

\begin{align*}
\sum_{Q_{k_v}' \subset Q} \ell(Q_{k_v}')^\beta \leq 2\ell(Q)^\beta\text{ for each dyadic cube } Q \in \mathcal{D}(Q_0),
\end{align*}
 
\begin{align}\label{goodpropertiesoftildeQ}   
\mathcal{H}^{\beta, Q_0}_\infty \left( \bigcup_k 
 Q_k' \right)
      &\leq \sum_{v, Q_{k_v}' \not \subseteq  \Tilde{Q}_k} \ell(Q_{k_v}')^\beta + \sum_k \ell(\Tilde{Q}_k)^\beta \\
      & \leq  2 \mathcal{H}^{\beta, Q_0}_\infty \left( \bigcup_k 
Q_k' \right).\nonumber
\end{align}
In particular, let $Q' \in \{Q_{k_v}' \}_v \cup \{\Tilde{Q}_k \}_k$, we have $Q'$ contains a cube in $\mathcal{A}_2 \subseteq \{Q_k^{2^{-\beta -2} t} \}_k$ and thus
\begin{align*}
    Q' \not \subseteq \left\{ x: \mathcal{M}^{\#}_{\beta, Q_0}f(x)>\frac{t}{A} \right\}.
\end{align*}
It then follows from $(B_2)$,  $(B_3)$ and definition of $\mathcal{M}^{\#}_{\beta, Q_0}f$ that
\begin{align}\label{estimatearg1}
    \frac{1}{\ell(Q')^\beta}\int_{Q'} |f| \; d\mathcal{H}^{\beta, Q_0}_\infty \leq \frac{t}{4}
\end{align}
and 
\begin{align}\label{estimaearg2}
    \frac{1}{\ell(Q')^\beta} \int_{Q'} |f - c_{Q'}|\; d\mathcal{H}^{\beta, Q_0}_\infty \leq \min \left( \frac{t}{A}, \frac{t}{4} \right),
\end{align}
where 
\begin{align*}
    c_{Q'}: = \argmin_{c\in \mathbb{R}}  \frac{1}{\ell(Q')^\beta} \int_{Q'} |f - c|\; d\mathcal{H}^{\beta, Q_0}_\infty.
\end{align*}
Moreover, the combination of \eqref{subadditivitydyadic}, (\ref{estimatearg1}) and (\ref{estimaearg2}) gives
\begin{align*}
    |c_{Q'}| &= \frac{1}{\ell(Q')^\beta} \int_{Q'} |c_{Q'}| \; d\mathcal{H}^{\beta, Q_0}_\infty\\
    &\leq \frac{1}{\ell(Q')^\beta} \int_{Q'} |c_{Q'} -f |\; d\mathcal{H}^{\beta, Q_0}_\infty +  \frac{1}{\ell(Q')^\beta} \int_{Q'} |f |\; d\mathcal{H}^{\beta, Q_0}_\infty \\
    & \leq \frac{t}{4}+ \frac{t}{4}\\
    &=\frac{t}{2}.
\end{align*}
Thus, if $Q_k^t \in \mathcal{A}_2$ and $Q^t_k \subseteq  Q'$, then 
\begin{align}\label{meanvalueestimateinsideQt1}
    \int_{Q^t_k} |f(x) - c_{Q'}| \; d\mathcal{H}^{\beta, Q_0}_\infty &\geq \int_{Q^t_k} |f| \;d \mathcal{H}^{\beta, Q_0}_\infty - c_{Q'} \ell(Q^t_k)^\beta\\ \nonumber
    &\geq t \ell(Q_k^t)^\beta - \frac{t}{2} \ell(Q_k^t)^\beta \\
    &= \frac{t}{2} \ell(Q_k^t)^\beta.\nonumber
 \end{align}
Now, we apply Lemma \ref{packing2} and (\ref{additionalpropertyforpacking2}) to $\mathcal{A}_2$ and obtain a non-overlapping subfamily $\{Q_{k_j}^t \}_j$ of $\mathcal{A}_2$ satisfying
\begin{align}\label{packingconditionofa2}
    \sum_{Q_{k_j}^t \subseteq Q} \ell(Q_{k_j}^t)^\beta \leq 2 \ell(Q)^ \beta \text{ for each dyadic cube }Q \in \mathcal{D}(Q_0)
\end{align}
and 
\begin{align}\label{packingcoditionofA2}
    \mathcal{H}^{\beta, Q_0}_\infty \left(\bigcup_{Q_k^t \in \mathcal{A}_2} Q_k^t \right) \leq  \sum_v \ell(Q_{k_j}^t)^\beta.
\end{align}
Now combining estimates (\ref{meanvalueestimateinsideQt1}) and (\ref{packingcoditionofA2}), we get
\begin{align}\label{final2}
    \mathcal{H}^{\beta, Q_0}_\infty\left(\bigcup_{Q_k^t \in \mathcal{A}_2} Q_k^t \right) & \leq   \sum_j \ell(Q_{k_j}^t)^\beta \\ \nonumber
    &= \sum_v \sum_{Q_{k_j}^t \subseteq Q'_{k_v} }  \ell(Q_{k_j}^t)^\beta + \sum_k \sum_{Q_{k_j}^t \subseteq\Tilde{Q}_k  }  \ell(Q_{k_j}^t)^\beta  \\ \nonumber
     & \leq \frac{2}{t} \sum_v \sum_{Q_{k_j}^t \subseteq Q'_{k_v} }  \int_{Q'_{k_v}} |f - c_{Q'_{k_v}}|\; d \mathcal{H}^{\beta, Q_0}_\infty \\ \nonumber
    & + \frac{2}{t} \sum_k \sum_{Q_{k_j}^t \subseteq\Tilde{Q}_k  }  \int_{\Tilde{Q}_k 
} |f - c_{\Tilde{Q}_k 
}|\; d \mathcal{H}^{\beta, Q_0}_\infty, \nonumber\\ \nonumber
    & \leq \frac{4}{t} \sum_v \int_{Q'_{k_v}} |f - c_{Q'_{k_v}}|\; d \mathcal{H}^{\beta, Q_0}_\infty \\ \nonumber
    & + \frac{4}{t} \sum_k \int_{\Tilde{Q}_k 
} |f - c_{\Tilde{Q}_k 
}|\; d \mathcal{H}^{\beta, Q_0}_\infty, \nonumber
\end{align}
where the last inequality follows the packing condition (\ref{packingconditionofa2}) of $\{ Q_{k_j}^t\}_j$ and (\ref{packingestimate}). Finally, using the fact that $\{ Q_k'\}_k$ is a subfamily of $\{Q_k^{2^{-\beta -2}t } \}_k$ and estimates (\ref{goodpropertiesoftildeQ}), (\ref{estimaearg2}) and (\ref{final2}), we obtain
\begin{align*}
    \mathcal{H}^{\beta, Q_0}_\infty \left(\bigcup_{Q_k^t \in \mathcal{A}_2} Q_k^t \right) & \leq \frac{4}{t} \sum_v \int_{Q'_{k_v}} |f - c_{Q'_{k_v}}|\; d \mathcal{H}^{\beta, Q_0}_\infty \\ \nonumber
    & + \frac{4}{t} \sum_k \int_{\Tilde{Q}_k 
} |f - c_{\Tilde{Q}_k 
}|\; d \mathcal{H}^{\beta, Q_0}_\infty\\
    &\leq \frac{4}{t} \frac{t}{A} \left( \sum_v \ell(Q'_{k_v})^\beta + \sum_k \ell(\Tilde{Q}_k )^\beta \right)\\
    & \leq \frac{8}{A} \mathcal{H}^{\beta, Q_0}_\infty \left(\bigcup_k  Q_k' \right)\\
    &\leq \frac{8}{A}\mathcal{H}^{\beta, Q_0}_\infty \left(  \bigcup_k Q_k^{2^{-\beta-2} t} \right) \\
    &= \frac{8}{A} \mu ( 2^{-\beta-2}t),
\end{align*}
   which completes the proof.
\end{proof}

We next give the proof of Theorem \ref{Maintheorem3}.

\begin{proof}[Proof of Theorem \ref{Maintheorem3}]
    Let $Q_0 = [0,1] \times \dots \times [0,1] \subset \mathbb{R}^d$. By Lemma \ref{equivoftwomaximal}, we deduce there exist $C_1, C_2, C_3 >0$ such that 
 \begin{align*}
     \int_{\mathbb{R}^d} \left( \mathcal{M}_{\mathcal{H}^\beta_\infty} f \right)^p \;d \mathcal{H}^{\beta}_\infty & = p \int^\infty_0 t^{p-1} \mathcal{H}^\beta_\infty\left( \left\{  x \in \mathbb{R}^d : \mathcal{M}_{\mathcal{H}^{\beta}_\infty} f (x) >t \right\} \right) \;dt \\
     &\leq C_1 p \int^\infty_0 t^{p-1} \mathcal{H}^\beta_\infty\left( \left\{  x\in \mathbb{R}^d : \mathcal{M}_{\mathcal{H}^{\beta,Q_0}_\infty} f(x) > C_2   2^{-(\beta +d) } t\right\} \right)\; dt\\
     & = C_3  \int_{\mathbb{R}^d} \left( \mathcal{M}_{\mathcal{H}^{\beta,Q_0}_\infty} f \right)^p \;d \mathcal{H}^{\beta, Q_0}_\infty.
 \end{align*}
Moreover, it is easy to see that 
 \begin{align*}
   \mathcal{M}^{\#}_{\beta, Q_0 }f (x) \leq C_\beta \mathcal{M}^{\# }_{\beta}f (x),
\end{align*}
where $C_\beta$ is the constant given in (\ref{subadditivitydyadic}).
Thus, the proof is complete once we show that there exists $C' =C' (p,d, \beta)$ such that 
\begin{align}\label{sustitionofQ0inCorollary}
    \int_{\mathbb{R}^d} \left( \mathcal{M}^{\beta, Q_0}_{\infty}  f \right)^p \; d\mathcal{H}^{\beta, Q_0}_\infty \leq C' \int_{\mathbb{R}^d} \left( \mathcal{M}^{\#}_{\beta, Q_0 }f \right)^p \; d\mathcal{H}^{\beta, Q_0}_\infty. 
\end{align}
To this end, we apply Lemma \ref{lemmaforcorollay3} and obtain that for any $t>0$ and $A>0$, 
\begin{align}\label{estimatefromLemma2.11}
    \mu(t) \leq \mathcal{H}^{\beta, Q_0}_\infty( \{ x \in \mathbb{R}^d: \mathcal{M}^{\#}_{\beta, Q_0 }f (x) > \frac{t}{A}\})+ \frac{16}{A} \mu( 2^{-\beta-2} t) ,
\end{align}
where
\begin{align*}
    \mu(t)  =\mathcal{H}^{\beta, Q_0}_\infty \left( \left\{x \in \mathbb{R}^d: \mathcal{M}_{\mathcal{H}^{\beta,Q_0}_\infty} f(x) >t \right\} \right).
\end{align*}
In particular, for any $N>0$, using (\ref{estimatefromLemma2.11}) and change of variables, we obtain
\begin{align}\label{estimatelimittoN}
    p \int^N_0 t^{p-1} \mu (t) \; dt & \leq   p \int^N_0 t^{p-1} \mathcal{H}^{\beta, Q_0}_\infty( \{ x \in \mathbb{R}^d: \mathcal{M}^{\#}_{\beta, Q_0 }f (x) > \frac{t}{A}\})\;dt \\
    &+  \frac{8}{A} p \int^N_0 t^{p-1}  \mu( 2^{-\beta-2} t) \;dt \nonumber\\
    & \leq A^p   p \int^\infty_0 t^{p-1} \mathcal{H}^{\beta, Q_0}_\infty( \{ x \in \mathbb{R}^d: \mathcal{M}^{\#}_{\beta, Q_0 }f (x) > t\})\;dt \nonumber \\
    &+ 2^{(\beta+2)p} \frac{8}{A} p \int^N_0 t^{p-1}  \mu(  t) \;dt \nonumber\\
    &= A^p \int_{\mathbb{R}^d} \left( \mathcal{M}^{\#}_{\beta, Q_0 }f \right)^p \; d\mathcal{H}^{\beta, Q_0}_\infty  + 2^{(\beta+2)p} \frac{8}{A} p \int^N_0 t^{p-1}  \mu(  t) \;dt. \nonumber
\end{align}
Moreover, we observe that
\begin{align*}
      \int^N_0 t^{p-1}  \mu(  t) \;dt  \leq M_N   \int^N_0 t^{p-p_0-1} \;dt =   M_N   \frac{N^{p-p_0}}{p-p_0} < \infty,
\end{align*}
where $M_N:=\sup\limits_{0<t <N} t^{p_0} \mu(t) < \infty$ by  \eqref{newassumptionThm1.2}.
Therefore, we may subtract both sides of inequality (\ref{estimatelimittoN}) with $2^{(\beta+2)p} \frac{8}{A} p \int^N_0 t^{p-1}  \mu(  t) \;dt$ and obtain 
\begin{align}\label{estimateforN2}
   \left( 1- 2^{(\beta+2)p} \frac{8}{A} \right) p \int^N_0 t^{p-1}  \mu(  t) \; dt \leq A^p \int_{\mathbb{R}^d} \left( \mathcal{M}^{\#}_{\beta, Q_0 }f \right)^p \; d\mathcal{H}^{\beta, Q_0}_\infty .
\end{align}
By choosing $A = \frac{2^{(\beta+2)p}}{16}$ and letting $N \to \infty$ in inequality (\ref{estimateforN2}), we obtain
\begin{align*}
   \frac{1}{2} \int_{\mathbb{R}^d} \left( \mathcal{M}_{\mathcal{H}^{\beta,Q_0}_\infty} f \right)^{p} \;d \mathcal{H}^{\beta,Q_0}_\infty & = \lim_{N \to \infty} \left(1- \frac{1}{2} \right)  p \int^N_0 t^{p-1}  \mu(  t) \;dt\\
   &\leq \frac{2^{(\beta+2)p^2}}{16^p}\int_{\mathbb{R}^d} \left( \mathcal{M}^{\#}_{\beta, Q_0 }f \right)^p \; d\mathcal{H}^{\beta, Q_0}_\infty
\end{align*}
by monotone convergence theorem. This completes the proof of (\ref{sustitionofQ0inCorollary}).

To prove \eqref{Maincorollay3-2}, it is sufficient to show that there exist $C' =C' (p,d, \beta)$ such that 
\begin{align}\label{sustitionofQ0inCorollary weak}
    \left\| \mathcal{M}^{\beta, Q_0}_{\infty}  f \right\|_{L^{p,\infty}(\mathcal{H}^\beta_\infty)} \le C' \left\| \mathcal{M}^{\#}_{\beta, Q_0 }f   \right\| _{L^{p,\infty}(\mathcal{H}^\beta_\infty)}.
\end{align}
Let $N>0$, using again \eqref{estimatefromLemma2.11}, multiplying by $t^p$ and taking the supremum over $0<\lambda\le N$ we obtain the following estimate,
    \begin{align*}
        \sup_{0<t\le N} t^p \mu(t) \le & \frac{16}{A} \sup_{0<t \le N} t^p  \mu(2^{-\beta-2}t) +  \sup_{0<t \le N} t^p  \mathcal{H}^{\beta, Q_0}_\infty(\{ x \in \mathbb{R}^d: \mathcal{M}^{\#}_{\beta, Q_0 }f (x) > \frac{t}{A}\})\\
        = & \frac{2^{(\beta+2)p+4}}{A} \sup_{0<t \le 2^{-\beta-2}N} t^p  \mu(t) + A^p \sup_{0<t \le \frac{N}{A}} t^p  \mathcal{H}^{\beta, Q_0}_\infty(\{ x \in \mathbb{R}^d: \mathcal{M}^{\#}_{\beta, Q_0 }f (x) > t\})\\
        \le & \frac{2^{(\beta+2)p+4}}{A} \sup_{0<t \le N} t^p  \mu(t) + A^p \left\| \mathcal{M}^{\#}_{\beta, Q_0 }f   \right\| _{L^{p,\infty}(\mathcal{H}^\beta_\infty)} ,
    \end{align*}
    for all $A>0$. Observe that the first term is finite,
    \begin{equation*}
        \sup_{0<t \le N} t^p  \mu(t) \le M_N N^{p-p_0}<\infty,
    \end{equation*}
where $M_N:=\sup\limits_{0<t <N} t^{p_0} \mu(t) < \infty$ by  \eqref{newassumptionThm1.2}. We take $A>0$ such that $\frac{2^{(\beta+2)p+4}}{A}= \frac{1}{2}$, that is $A=2^{(\beta+2)p+5} $. Then, since the right-hand side is finite, we can absorb it into the left-hand side, and we obtain,
\begin{equation*}
    \sup_{0<t\le N} t^p \mu(t) \le  2^{((\beta+2)p+5)p+1} \left\| \mathcal{M}^{\#}_{\beta, Q_0 }f   \right\| _{L^{p,\infty}(\mathcal{H}^\beta_\infty)}. 
\end{equation*}
We conclude the proof of \eqref{sustitionofQ0inCorollary weak} letting $N\to \infty$.
\end{proof}

\section{pointwise estimate of fractional maximal function and sharp maximal function of $I_\alpha f$}

In this section, we will first give the proof of Theorem \ref{Adamsextend1} and then provide an analogous of Theorem \ref{wheedenHausdorffcontent} based on Adams's approach. We begin with a result which can be found in \cite[Lemma 2.2]{chen2023capacitary} and \cite[Proposition 3.6]{harjulehto2024hausdorff} and is an extension of \cite[Lemma 3]{OV}.

\begin{lemma}\label{alphabetachange}
    Let $f \geq 0$ and $0 < \alpha \leq \beta \leq d \in \mathbb{N}$. Then 
    \begin{align*}
        \int_{\mathbb{R}^d} f \; d\mathcal{H}^\beta_\infty \leq  \frac{\beta }{\alpha} \left(  \int_{\mathbb{R}^d} f^{\frac{\alpha}{\beta}}  \; d\mathcal{H}^\alpha_\infty \right)^{\frac{\beta}{\alpha}}.
    \end{align*}
\end{lemma}
With Lemma \ref{alphabetachange}, one can immediately establish the following pointwise estimate for the $\beta$-dimensional sharp maximal function.
\begin{lemma}\label{sharpalphabetachange}
    Let $0 < \alpha \leq \beta \leq d \in \mathbb{N} $. Then there exists a constant $C>0$ depending only on $\alpha$ and $\beta$ such that 
    \begin{align*}
        \mathcal{M}^{\#} _\beta f (x) \leq C \mathcal{M}^{\#} _\alpha f  (x)
    \end{align*}
    for all $x \in \mathbb{R}^d$.
\end{lemma}
\begin{proof}
    Let $0 < \alpha \leq \beta \leq d $. Fix $x \in \mathbb{R}^d$ and $Q$ be a cube containing $x$. An application of Lemma \ref{alphabetachange} and H\"{o}lder's inequality in Lemma \ref{basicChoquet} (iii) yields that
    \begin{align*}
        \frac{1}{\mathcal{H}^\beta_\infty(Q)} \int_Q |f -c_Q | \; d \mathcal{H}^{\beta}_\infty & \leq  \frac{\beta}{\alpha} \frac{1}{\mathcal{H}^\beta_\infty(Q) } \left(  \int_Q |f-c_Q |^{\frac{\alpha}{\beta}} \; d\mathcal{H}^{\alpha}_\infty \right)^\frac{\beta}{\alpha}\\
        &\leq   \frac{2 \beta \mathcal{H}^\alpha_\infty(Q)^{\frac{\beta}{\alpha}}}{ \alpha \mathcal{H}^\beta_\infty(Q)}  \left(\frac{1}{\mathcal{H}^\alpha_\infty(Q) }  \int_Q |f-c_Q |\; d\mathcal{H}^{\alpha }_\infty \right)\\
        &\leq C \mathcal{M}^{\#} _\alpha f (x),
    \end{align*}
    where
    \begin{equation*}
        c_Q= \argmin_{c\in \mathbb{R}}  \frac{1}{\ell(Q)^\alpha} \int_{Q} |f - c_{Q}|\; d\mathcal{H}^{\alpha}_\infty
    \end{equation*}
and 
\begin{align*}
    C =\frac{2 \beta \mathcal{H}^\alpha_\infty(Q)^{\frac{\beta}{\alpha}}}{ \alpha \mathcal{H}^\beta_\infty(Q)}
\end{align*}
does not depend on $Q$.  Taking the supremum over all cubes $Q$ such that $x\in Q$, we obtain 
    \begin{align*}
        \mathcal{M}^{\#} _\beta f (x) \leq C   \mathcal{M}^{\#} _\alpha f (x).
    \end{align*}
\end{proof}

We next record two results from \cite[Lemma 3.5 and Lemma 3.8]{chen2024selfimproving}.

\begin{lemma}\label{forcompactmorreya}
    Let $0< \alpha<d \in \mathbb{N}$, $\beta\in (d-\alpha, d]$, $Q$ be an open cube in $\mathbb{R}^d$ with centre $x_0$ and $f$ be a real-valued measurable function in $\mathbb{R}^d$. If $\supp{f} \subseteq 2Q$, then $I_\alpha f \in L^1(Q;\mathcal{H}^{\beta}_\infty)$ and there exists a constant $C$ depending only on $d$, $\alpha $ and $\beta$ such that
    \begin{align}
        \int_Q | I_\alpha f  | \; d \mathcal{H}^{\beta}_\infty \leq C \ell(Q)^{\beta} \mathcal{M}_\alpha f (x_0).
    \end{align}
\end{lemma}

\begin{lemma}\label{outsidef2inMorrey}
    Let $0<\alpha < d \in \mathbb{N}$, $\beta \in (0,d]$, $Q$ be an open cube in $\mathbb{R}^d$ with centre $x_0$ and $f$ be a real-valued measurable function with $\supp{f} \subseteq (2Q)^c$. If  $I_\alpha f \in L^1_{loc} (\mathbb{R}^d)$, then $I_\alpha f $ is continuous in Q and there exists $c \in \mathbb{R}$ such that 
    \begin{align}\label{Morreyoutsideestimate}
\int_{Q} | I_\alpha f -c| \; d\mathcal{H}^\beta_\infty \leq C \ell(Q)^\beta   \mathcal{M}_\alpha f (x_0).
    \end{align}
 
\end{lemma}

The combination of Lemma \ref{forcompactmorreya} and \ref{outsidef2inMorrey} readily yields the following pointwise estimate concerning $\beta$-dimensional sharp maximal function and fractional maximal function.

\begin{theorem}
 \label{forcompactmorrey}
    Let $0< \alpha<d \in \mathbb{N}$ and $\beta \in (d-\alpha, d]$. If $I_\alpha f \in L^1_{loc}(\mathbb{R}^d)$, then  for every $x \in \mathbb{R}^d$, there exists a constant $C$ depending only on $d$, $\alpha $, $\beta$ such that
    \begin{align}
       \mathcal{M}^{\#}_{\beta}(I_\alpha f) (x) \leq C \mathcal{M}_\alpha f (x).
    \end{align}
\end{theorem}

For the proof of Theorem \ref{forcompactmorrey}, we introduce $\beta$-dimensional centered sharp maximal operator $\mathcal{M}^{\#, c}_{\beta}f(x)  $ of a function $f$ which is defined as
\begin{align*}
    \mathcal{M}^{\#, c}_{\beta}f(x) := \sup_Q \inf_{c \in \mathbb{R}} \frac{1}{\ell(Q)^\beta } \int_Q |f -c | \; d\mathcal{H}^\beta_\infty,
\end{align*}
where the supremum is taken over all cubes $Q$ with center $x$. It is easy to see that there exists a constant $C_\beta >0$ such that
\begin{align}\label{equivalentcentreuncentre}
    \mathcal{M}^{\#, c}_{\beta}f(x)  \leq \mathcal{M}^{\#}_{\beta}f(x)  \leq C_\beta \mathcal{M}^{\#, c}_{\beta}f(x) .
\end{align}

\begin{proof}[Proof of Theorem \ref{forcompactmorrey}]
    Suppose that $I_\alpha f \in L^1_{loc} (\mathbb{R}^d)$. Let $Q$ be a cube in $\mathbb{R}^d$ with centre $x_0$ and $f_1 = f \chi_{2Q}$, $f_2 = f -f_1$. Then there exists a constant $c \in \mathbb{R}$ and $C$ depending on $d$, $\alpha $, $\beta$ such that
    \begin{align*}
        \int_Q |I_\alpha f -c| \; d\mathcal{H}^{\beta}_\infty &=  \int_Q |I_\alpha f_1 + I_\alpha f_2 -c| \; d\mathcal{H}^{\beta}_\infty\\
        &\leq 2 \left( \int_Q |I_\alpha f_1 | \; d\mathcal{H}^{\beta}_\infty + \int_Q |I_\alpha f_2 -c| \; d\mathcal{H}^{\beta}_\infty\right)\\
        & \leq 2C \ell(Q)^{\beta} \mathcal{M}_\alpha f (x_0),
    \end{align*}
    where we use Lemma \ref{forcompactmorreya} and Lemma \ref{outsidef2inMorrey} in the last inequality.
    Now dividing both sides with $\ell(Q)^{\beta}$ and taking the supremum over all cubes $Q$ with centre $x$, we obtain
    \begin{align*}
        \mathcal{M}^{\#, c}_{\beta}(I_\alpha f)(x)\leq C \mathcal{M}_\alpha f (x)
    \end{align*}
    and thus 
    \begin{align*}
       \mathcal{M}^{\#}_{\beta} (I_\alpha f) (x)\leq C \mathcal{M}_\alpha f (x)
    \end{align*}
by (\ref{equivalentcentreuncentre}).   
\end{proof}


We now give the proof of Theorem \ref{Adamsextend1}

\begin{proof}[Proof of Theorem \ref{Adamsextend1}]
Let $x \in \mathbb{R}^d$. By Lemma \ref{forcompactmorrey}, we have
\begin{align*}
     \mathcal{M}^{\#}_{\beta } (I_\alpha f) (x) \lesssim \mathcal{M}_\alpha f (x).
\end{align*}
The combination of (\ref{Adamspointwise}) and  Lemma \ref{sharpalphabetachange} yields that
\begin{align*}
    \mathcal{M}_\alpha f (x) \cong \mathcal{M}^{\#}(I_\alpha f) (x) \lesssim \mathcal{M}^{\#}_{\beta } (I_\alpha f) (x),
\end{align*}
where we note that $\mathcal{M}_n^\# \simeq \mathcal{M}^\#$ (see for instance \cite[Remark 3.4]{harjulehto2024hausdorff}). This completes the proof.
\end{proof}

\section{Good lambda estimate for Riesz potential with respect to Hausdorff content and applications}

In this section, we present the proof of Theorem \ref{wheedenHausdorffcontent}. To this end, we first prove a good lambda inequality with exponential decay, following the ideas of \cite{HonzikJaye} and using the packing condition to avoid difficulties due to the nonlinearity of the Hausdorff content. We conclude the section with some consequences of the exponential decay on the good-lambda estimate.

First, we define a dyadic version of the Riesz potential, which we will 
use to approximate $I_\alpha$.  We begin with the definition of a dyadic grid. A dyadic grid $\mathcal{D}$ is a countable family of cubes with the following properties
\begin{enumerate}
    \item For each $Q\in \mathcal{D}$, $\ell(Q)=2^k$ for some $k\in \Z$.
    \item Given $Q,P\in \mathcal{D}$, then $Q\cap P \in \{ \emptyset, P, Q\}$.
    \item For each $k\in \Z$, the set $\mathcal{D}_k= \{Q\in \mathcal{D}: \ell(Q)=2^k\}$ is a partition of $\R^d$. 
\end{enumerate}

The following dyadic operator will play a central role in the proof of the good lambda estimate. 

\begin{definition}
    Let $0<\alpha <d\in \mathbb{N}$ and let $\mathcal{D}$ a dyadic grid. We define the dyadic Riesz potential with respect to $\mathcal{D}$ of a locally finite measure $\mu$, by
    \begin{equation*}
       I_\alpha^\mathcal{D} \mu (x)= \sum _{ Q\in \mathcal{D}} \frac{\mu(Q)}{\ell(Q)^{d-\alpha}}\chi_Q(x).
    \end{equation*}
\end{definition}

Dyadic Riesz potentials were first introduced by E. Sawyer and R. L. Wheeden \cite{SawyerWheeden}. We will use the following result by D. Cruz-Uribe and K. Moen (see \cite[Proposition 2.2]{CruzUribeMoen}) to approximate the Riesz potential by a linear combination of dyadic Riesz potential.

\begin{proposition}\label{PropCruzUribeMoen}
    Given $0<\alpha <d \in \mathbb{N}$, and a locally finite measure $\mu$, then for any dyadic grid $\mathcal{D}$, 
    \begin{equation}\label{dyadicequiv1}
        I_\alpha^\mathcal{D} \mu (x) \le c_{d, \alpha} I_\alpha \mu (x).
    \end{equation}
    Conversely, we have that
    \begin{equation*}
        I_\alpha \mu (x) \le c_{d, \alpha} \max_{i\in \{ 0, \frac{1}{3}\}^d } I_\alpha^{\mathcal{D}^i} \mu (x),
    \end{equation*}
    where
    \begin{equation*}
    \mathcal{D}^i = \{ 2^{-k}([0,1)^d + m + (-1)^k i ) : k\in \Z, m\in \Z^d \}, \qquad i\in \{ 0, \frac{1}{3}\}^d. 
\end{equation*}
    Hence, the Riesz potential is pointwise equivalent to a linear combination of dyadic Riesz potentials, that is,
    \begin{equation}\label{pointwiseequiv}
        I_\alpha \mu (x) \cong_{d,\alpha} \sum_{i\in \{ 0, \frac{1}{3}\}^d }I_\alpha^{\mathcal{D}^i} \mu (x) .
    \end{equation}
\end{proposition}

We should mention that the proof of the previous result in \cite{CruzUribeMoen} is made for functions, but a detailed inspection of the proof shows that the same proof works for measures.

Now, we state the main result of this section. The result is an extension of the good lambda inequality proved in \cite{HonzikJaye}, which improves the classical result of B. Muckenhoupt and R. L. Wheeden \cite{MR0340523}.

\begin{theorem}\label{goodlambdaestimaterelatedRieszandfractional}
    Let $0<\alpha <d$, $\beta\in (d-\alpha, d]$ and let $\mathcal{D}$ be a dyadic grid. Then, there exist constants $C,c>0$ depending on $d, \alpha, \beta$ such that 
    \begin{align}\label{good lambda exp}
    \begin{split}
        \mathcal{H}^\beta_\infty(\{ x\in \R^d :  I_\alpha^\mathcal{D} \mu (x)>2\lambda, \mathcal{M}_\alpha \mu(x) \le \varepsilon \lambda \}) \le & C e^{-\frac{c}{\varepsilon}}  \mathcal{H}^\beta_\infty(\{ x\in \R^d :  I_\alpha^\mathcal{D} \mu (x)>\lambda\}) 
    \end{split}
    \end{align}
    for all $\lambda>0$ and all $0<\varepsilon<1$ and all locally finite measure $\mu$.
\end{theorem}

\begin{proof}
   It is sufficient to prove \eqref{good lambda exp} for the dyadic content $\mathcal{H}^{\beta, Q_0}_\infty$, where $Q_0$ is a cube in $\mathcal{D}$, since both contents are equivalent by \eqref{equivalenceofdyadic}.  
   
   Fix $\lambda>0$ and $0<\varepsilon< 1$. We denote $G_\lambda= \{ x\in \R^d :   I_\alpha^\mathcal{D} \mu(x)>\lambda\}$. We can assume without loss of generality $\mathcal{H}^{\beta, Q_0}_\infty (G_\lambda)<\infty$. Let $\{Q_j^\lambda\}_j$ the maximal collection of dyadic cubes subordinates to $\mathcal{D}$ contained in $G_\lambda$ with respect to the inclusion.
   

    Applying Lemma \ref{packing2} to the family $\{Q_{j}^\lambda\}_j$ we obtain a subfamily $\{Q_{j_{k}}^\lambda\}_k$ and a family of non-overlapping ancestors $\{\tilde{Q}_k\}_k$ such that
    \begin{enumerate}
        \item \begin{align*}
G_\lambda = \bigcup_{k} Q_{j}^\lambda \subset \bigcup_{k} Q_{j_{k}}^\lambda \cup \bigcup_{k} \tilde{Q}_k
\end{align*}
\item
\begin{align*}
\sum_{Q_{j_{k}}^\lambda \subset Q}\ell(Q_{j_{k}}^\lambda)^\beta \leq 2\ell (Q)^\beta, \text{ for each dyadic cube } Q.
\end{align*}
\item
 \begin{align*}
      \mathcal{H}^{\beta, Q_0}_\infty(\cup_j Q_{j}^\lambda) 
            &\leq \sum_{k, Q_{j_{k}}^\lambda \not \subseteq \tilde{Q}_m} \ell(Q_{j_{k}}^\lambda)^\beta + \sum_k \ell(\Tilde{Q}_k)^\beta\\
            & \leq 2\sum_k \ell(Q_{j_{k}}^\lambda)^\beta\\
      & \leq  2  \mathcal{H}^{\beta, Q_0}_\infty (\cup_j Q_{j}^\lambda).
  \end{align*}
    \end{enumerate}

   Define $S=\{ x\in \R^d :  I_\alpha^\mathcal{D} \mu(x) > 2\lambda, \mathcal{M}_\alpha \mu (x)\le \varepsilon \lambda \} $.  Since $G_{2\lambda}\subseteq G_\lambda$, we observe that
    \begin{align*}
    	S \subseteq & G_{2\lambda}
    	\subseteq  G_\lambda  
    	=  \bigcup_{k} Q_{j_{k}}^\lambda \cup \bigcup_{k} \tilde{Q}_k,
    \end{align*}
    and then we have
    \begin{align*}
    	S \subseteq & \bigcup_{k: Q_{j_k}^\lambda \not \subseteq \tilde{Q}_m} \{ x\in Q_{j_k}^\lambda :  I_\alpha^\mathcal{D} \mu(x) > 2\lambda, \mathcal{M}_\alpha \mu (x)\le \varepsilon \lambda \}\\
    	& \cup \bigcup_k \{ x\in \tilde{Q}_k :  I_\alpha^\mathcal{D} \mu(x) > 2\lambda, \mathcal{M}_\alpha \mu (x)\le \varepsilon \lambda \}.  
    \end{align*}
    
    Using the subadditivity of $\mathcal{H}_\infty^{\beta,Q_0}$ we obtain
    \begin{align}
    	\begin{split}
    		\mathcal{H}_\infty^{\beta,Q_0} ( S ) \le & \sum_{k: Q_{j_k}^\lambda \not \subseteq \tilde{Q}_m} \mathcal{H}_\infty^{\beta,Q_0} (\{ x\in Q_{j_k}^\lambda :  I_\alpha^\mathcal{D} \mu(x) > 2\lambda, \mathcal{M}_\alpha \mu (x)\le \varepsilon \lambda \})\\ 
    		&+\sum_{k} \mathcal{H}_\infty^{\beta,Q_0} (\{ x\in \tilde{Q}_k :  I_\alpha^\mathcal{D} \mu(x) > 2\lambda, \mathcal{M}_\alpha \mu (x)\le \varepsilon \lambda \}).
    	\end{split}
    \end{align}

Fix $Q'\in \{Q_{j_{k}}^\lambda\}_k\cup \{\tilde{Q}_k\}_k $, and let us define $S_{Q'}= \{ x\in Q' :  I_\alpha^\mathcal{D} \mu(x) > 2\lambda, \mathcal{M}_\alpha \mu (x)\le \varepsilon \lambda \}$. We can assume that $S_{Q'}$ is non-empty.  

We have the following estimate on the tail of the potential
\begin{equation*}
    \sum_{ \substack{Q\in \mathcal{D}\\ Q\supsetneq Q'}} \frac{\mu(Q)}{\ell(Q)^{d-\alpha}}\le \lambda,
\end{equation*}
by the definition of $G_\lambda$ and observing that either $Q'$ is a maximal cube of the family $\{Q_j^\lambda\}_j$, or it contains a maximal cube in $\{Q_j^\lambda\}_j$. Let $j_0\in \mathbb{Z}$ such that $\ell(Q')=2^{j_0}$ and $m\in \mathbb{N}$ to be chosen later. If $x\in S_{Q'}$, we have
\begin{align*}
     I_\alpha^\mathcal{D} \mu(x) = & \sum_{x\in Q\in \mathcal{D}} \frac{\mu(Q)}{\ell(Q)^{d-\alpha}} \\
    = & \sum_{ \substack{x\in Q\in \mathcal{D} \\ Q\supsetneq Q'} } \frac{\mu(Q)}{\ell(Q)^{d-\alpha}} + \sum_{\substack{x\in Q\in \mathcal{D}\\ 2^{j_0-m}<\ell(Q) \le 2^{j_0}}} \frac{\mu(Q)}{\ell(Q)^{d-\alpha}} + \sum_{\substack{x\in Q\in \mathcal{D} \\ \ell(Q)\le 2^{j_0-m}}} \frac{\mu(Q)}{\ell(Q)^{d-\alpha}}\\
    \le &  \lambda + m \mathcal{M}_\alpha \mu(x)  +\sum_{\substack{x\in Q\in \mathcal{D} \\ \ell(Q)\le 2^{j_0-m}}} \frac{\mu(Q)}{\ell(Q)^{d-\alpha}}\\ 
    \le & \lambda + m \varepsilon \lambda +\sum_{k\le j_0-m} g_k(x),
\end{align*}
where $g_k(x)$ is defined as
\begin{equation*}
    g_k(x)=\sum_{x\in Q \in \mathcal{D}_k} \frac{\mu(Q)}{\ell(Q)^{d-\alpha}} \quad \text{ for } k\in \Z
\end{equation*}
and $\mathcal{D}_k= \{Q\in \mathcal{D} : \ell(Q)=2^k\}$.  Observe that only one nonzero term exists for each $x$ in the previous sum.

Let $\eta= \frac{\beta-(d-\alpha)}{2}$, we remark that since $\beta\in (d-\alpha, d]$ we have $\eta>0$. We claim
\begin{equation}\label{aux1}
    S_{Q'}\subseteq  \bigcup_{k\le j_0-m} \{ x\in Q' : g_k(x)> \lambda(1-m\varepsilon)2^{\eta (k-j_0+m)} (1-2^{-\eta})\}.
\end{equation}

To this end, observe that if $x\in S_{Q'}$, we have
\begin{equation*}
    2\lambda<  I_\alpha^\mathcal{D} \mu(x) \le \lambda + m \varepsilon \lambda +\sum_{k\le j_0-m} g_k(x)
\end{equation*}
and thus,
\begin{equation}\label{eq gk }
    \lambda(1-m\varepsilon)<\sum_{k\le j_0-m} g_k(x).
\end{equation}

If we assume that 
\begin{equation*}
    g_k(x)\le \lambda(1-m\varepsilon)2^{\eta (k-j_0+m)} (1-2^{-\eta})
\end{equation*}
for all $k\le j_0-m$, then 
\begin{equation*}
    \sum_{k\le j_0-m} g_k(x)\le \lambda(1-m\varepsilon)(1-2^{-\eta})\sum_{k\le j_0-m}2^{\eta (k-j_0+m)} = \lambda(1-m\varepsilon),
\end{equation*}
which contradicts to \eqref{eq gk }. Thus, there exists $k\le j_0-m$ such that 
\begin{equation*}
    g_k(x)> \lambda(1-m\varepsilon)2^{\eta (k-j_0+m)} (1-2^{-\eta}),
\end{equation*}
and we conclude the proof of the claim. 

Observe that for any $\gamma>0$ and any $k\le j_0$, we have
\begin{align}\label{aux2}
 \nonumber    \mathcal{H}^{\beta, Q_0}_\infty ( \{ x\in Q' : g_k(x)>\gamma \}) \le & \frac{1}{\gamma} \int_{Q'} g_k \, d\mathcal{H}^{\beta, Q_0}_\infty \\
\nonumber     = & \frac{1}{\gamma} \sum_{\substack{x\in Q \in \mathcal{D}_k \\ Q\subseteq Q'} } \frac{\mu(Q)}{\ell(Q)^{d-\alpha}} \ell(Q)^\beta \\
 \nonumber    = & \frac{1}{\gamma} \sum_{\substack{x\in Q \in \mathcal{D}_k \\ Q\subseteq Q'} } \ell(Q)^{\beta-(d-\alpha)} \mu(Q) \\
  \nonumber   \le & \frac{1}{\gamma}  2^{2\eta k } \mu(Q') \\
\nonumber     = & \frac{1}{\gamma}  2^{2\eta k } \frac{1}{\ell(Q')^{2\eta}} \ell(Q')^\beta \frac{\mu(Q')}{\ell(Q')^{d-\alpha}} \\
\nonumber     \le & \frac{1}{\gamma}  2^{2\eta (k-j_0) }  \ell(Q')^\beta \inf_{z\in Q'} \mathcal{M}_\alpha \mu(z) \\
 \nonumber    \le & \frac{1}{\gamma}  2^{2\eta (k-j_0) }  \ell(Q')^\beta  \mathcal{M}_\alpha \mu(x_{Q'}) \\
    \le & \frac{1}{\gamma}  2^{2\eta (k-j_0) }  \ell(Q')^\beta  \varepsilon \lambda,
\end{align}
where $x_{Q'}\in Q'$ such that $\mathcal{M}_\alpha \mu(x_{Q'})\le \varepsilon \lambda$, that we know exists because otherwise, $S_{Q'}$ would be empty.  

The combination of (\ref{aux1}) and (\ref{aux2}) then yields that
\begin{align*}
    \mathcal{H}^{\beta, Q_0}_\infty ( S_{Q'}) \le & \sum_{k\le j_0-m} \mathcal{H}^{\beta, Q_0}_\infty ( \{ x\in Q' : g_k(x)> \lambda(1-m\varepsilon)2^{\eta (k-j_0+m)} (1-2^{-\eta})\}) \\
    \le & \frac{1}{\lambda(1-m\varepsilon) (1-2^{-\eta})} \ell(Q')^\beta \varepsilon \lambda  \sum_{k\le j_0-m} \frac{2^{2\eta (k-j_0) }}{2^{\eta (k-j_0+m)}} \\
    = & \ell(Q')^\beta  \frac{\varepsilon}{1-m\varepsilon }  \frac{2^{-2\eta m}}{1-2^{-\eta}}  \sum_{k\le j_0-m} 2^{-\eta (j_0-m-k) } \\
     = & \ell(Q')^\beta  \frac{\varepsilon}{1-m\varepsilon }  \frac{2^{-2\eta m}}{1-2^{-\eta}}  \sum_{k=0}^\infty 2^{-\eta k } \\
     = & C_\eta 2^{-2\eta m} \frac{\varepsilon}{1-m\varepsilon }  \ell(Q')^\beta.
\end{align*}
We remark that the previous sum is finite because $\eta>0$. Now we choose $m\in \mathbb{N}\cup\{0\}$ such that 
\begin{equation*}
    \frac{1}{\varepsilon}-2 \le m < \frac{1}{\varepsilon}-1,
\end{equation*}
then $\frac{\varepsilon}{1-m\varepsilon }<1$, and hence, there exists $c=c(d,\alpha,\beta)$ such that
\begin{equation*}
    \mathcal{H}^{\beta, Q_0}_\infty ( S_{Q'}) \le C_\eta 2^{-\frac{2\eta}{\varepsilon}}  \ell(Q')^\beta = C_\eta e^{-\frac{c}{\varepsilon}}  \ell(Q')^\beta.
\end{equation*}

Using the previous estimate and the packing condition of the family $\{Q_{j_k}^{\lambda}\}_k$, we have
\begin{align*}
    \mathcal{H}_\infty^{\beta,Q_0} ( S ) \le & \sum_{k: Q_{j_k}^\lambda \not \subseteq \tilde{Q}_m} \mathcal{H}_\infty^{\beta,Q_0} (S_{Q_{j_k}^\lambda}) +\sum_{k} \mathcal{H}_\infty^{\beta,Q_0} (S_{\tilde{Q}_k})\\
    \le &  C_\eta e^{-\frac{c}{\varepsilon}}  \left( \sum_{k: Q_{j_k}^\lambda \not \subseteq \tilde{Q}_m} \ell (Q_{j_k}^\lambda)^\beta  +\sum_{k} \ell( \tilde{Q}_k)^\beta   \right) \\
    \le & 2C_\eta e^{-\frac{c}{\varepsilon}} \mathcal{H}^{\beta, Q_0}_\infty \left( \bigcup_{j} Q_{j}^{\lambda}\right)   \\ 
    = & 2C_\eta e^{-\frac{c}{\varepsilon}} \mathcal{H}^{\beta, Q_0}_\infty \left(G_\lambda\right) .
\end{align*}
\end{proof}

As a consequence of the previous good-lambda estimate, we give the proof of Theorem \ref{wheedenHausdorffcontent}.

\begin{proof}[Proof of Theorem \ref{wheedenHausdorffcontent}]
    We observe that it is enough to prove \eqref{MW strong} and \eqref{MW weak} for the dyadic Riesz potential $I_\alpha^\mathcal{D}$ with a constant independent of the dyadic grid $\mathcal{D}$. To show this, we use \eqref{pointwiseequiv}, and we obtain the desired inequality,
    \begin{align*}
        \Vert I_\alpha \mu \Vert_{L^p (\mathcal{H}^\beta_\infty)} \leq & c_{d,\alpha} \left\| \sum_{i\in \{0,\frac{1}{3}\}^d } I_\alpha^{\mathcal{D}^i} \mu \right\|_{L^p (\mathcal{H}^\beta_\infty)}\\
        \le &  c_\beta c_{d,\alpha}  \sum_{i\in \{0,\frac{1}{3}\}^d } \Vert  I_\alpha^{\mathcal{D}^i} \mu \Vert_{L^p (\mathcal{H}^\beta_\infty)}\\
        \le & C p\,  c_\beta c_{d,\alpha}  \sum_{i\in \{0,\frac{1}{3}\}^d }  \Vert \mathcal{M}_\alpha \mu \Vert_{L^p(\mathcal{H}^{\beta}_\infty)}  \\
        = & C_{\beta, \alpha, d} p  \Vert \mathcal{M}_\alpha \mu \Vert_{L^p(\mathcal{H}^{\beta}_\infty)} . 
    \end{align*}
    Similarly, we can obtain \eqref{MW weak} from the dyadic estimate.
    
    Let us first assume that $\mu$ has compact support. We may also assume that $\Vert \mathcal{M}_\alpha \mu \Vert_{L^p(\mathcal{H}^{\beta}_\infty)}$ is finite because otherwise, there is nothing to prove. Let $\mathcal{D}$ a dyadic grid, using Theorem \ref{good lambda exp}, and the subadditivity of $ \mathcal{H}^\beta_\infty$ we obtain, 
    \begin{align*}
        \mathcal{H}^\beta_\infty( \{ x\in \R^d : I_\alpha^\mathcal{D} \mu (x)>2\lambda\})   & \leq \mathcal{H}^\beta_\infty( \{ x\in \R^d : I_\alpha^\mathcal{D} \mu(x)>2\lambda, \mathcal{M}_\alpha \mu(x)\le \varepsilon \lambda \}) \\
        &+   \mathcal{H}^\beta_\infty( \{ x\in \R^d : \mathcal{M}_\alpha \mu(x)> \varepsilon \lambda\})\\
        & \leq  C e^{-\frac{c}{\varepsilon}}  \mathcal{H}^\beta_\infty(\{ x\in \R^d : I_\alpha^\mathcal{D} \mu (x)>\lambda\}) \\
        & + C \mathcal{H}^\beta_\infty( \{ x\in \R^d : \mathcal{M}_\alpha \mu(x)> \varepsilon \lambda\}),
    \end{align*}
    for all $\lambda>0$ and all $0<\varepsilon<1$. Let $N>0$,  we have
    \begin{align*}
        \int_0^N \lambda^{p-1} \mathcal{H}^\beta_\infty( \{ x\in \R^d : I_\alpha^\mathcal{D} \mu (x)>\lambda\}) d\lambda & =   2^p \int_0^\frac{N}{2} \lambda^{p-1} \mathcal{H}^\beta_\infty( \{ x\in \R^d : I_\alpha^\mathcal{D} \mu(x) >2\lambda\}) d\lambda \\
         & \leq  2^p C e^{-\frac{c}{\varepsilon}} \int_0^\frac{N}{2} \lambda^{p-1} \mathcal{H}^\beta_\infty( \{ x\in \R^d : I_\alpha^\mathcal{D} \mu(x)>\lambda\}) d\lambda \\
        & +2^p C \int_0^\frac{N}{2} \lambda^{p-1} \mathcal{H}^\beta_\infty( \{ x\in \R^d : \mathcal{M}_\alpha \mu(x)>\varepsilon\lambda\}) d\lambda \\
          & \leq  2^p C e^{-\frac{c}{\varepsilon}} \int_0^N \lambda^{p-1} \mathcal{H}^\beta_\infty( \{ x\in \R^d : I_\alpha^\mathcal{D} \mu(x)>\lambda\}) d\lambda \\
        & +\frac{2^p C }{\varepsilon^p}\int_0^\frac{N\varepsilon}{2} \lambda^{p-1} \mathcal{H}^\beta_\infty( \{ x\in \R^d : \mathcal{M}_\alpha \mu(x)>\lambda\}) d\lambda
    \end{align*} 
    for all $0<\varepsilon<1$. Notice that the first term is finite because $\mu$ has compact support.  Let $B$ be a ball such that $\supp{\mu}\subseteq B$. If $x\notin 2B$, we have $r(B)>|x-y|$ for all $y\in B$. Hence, 
    \begin{equation*}
        I_\alpha \mu (x) = \frac{1}{\gamma(\alpha)}\int_{B} \frac{d\mu(y)}{|x-y|^{d-\alpha }} \le \frac{1}{\gamma(\alpha)} \frac{\mu(B)}{r(B)^{d-\alpha}} \le \frac{1}{\gamma(\alpha)} \mathcal{M}_\alpha \mu(x), 
    \end{equation*}
    for all $x\notin 2B$.  Using the previous estimate, we obtain
    \begin{align*}
        \{ x\in \R^d :  I_\alpha \mu (x)>\lambda \} = &   \{ x\in 2B :  I_\alpha \mu (x)>\lambda \} \cup  \{ x\in (2B)^c :  I_\alpha \mu (x)>\lambda \}\\
        \subseteq &  \{ x\in 2B :  I_\alpha \mu (x)>\lambda \} \cup  \{ x\in (2B)^c :  \mathcal{M}_\alpha \mu (x)>\gamma(\alpha) \lambda \},
    \end{align*}
    for all $\lambda>0$. Then, we have
    \begin{align*}
        \int_0^N \lambda^{p-1} \mathcal{H}^\beta_\infty( \{ x\in \R^d : I_\alpha \mu(x)>\lambda\}) d\lambda  & \leq  \int_0^N \lambda^{p-1} \mathcal{H}^\beta_\infty( \{ x\in 2B : I_\alpha \mu(x)>\lambda\}) d\lambda \\
        &+ \int_0^N \lambda^{p-1} \mathcal{H}^\beta_\infty( \{ x\in \R^d : \mathcal{M}_\alpha \mu(x)>\gamma(\alpha) \lambda\}) d\lambda\\
         & \leq  r(2B)^\beta \frac{N^p}{p} + \frac{1}{\gamma(\alpha)^p} \Vert \mathcal{M}_\alpha \mu \Vert_{L^p(\mathcal{H}^{\beta}_\infty)}^p \\
        & <  \infty,
    \end{align*}
    and using \eqref{dyadicequiv1} we obtain that the dyadic counterpart is also finite.

    We take $\varepsilon$ such that $2^p C e^{-\frac{c}{\varepsilon}} =\frac{1}{2}$, that is $\varepsilon=\frac{c}{\log (2^{p+1}C)}$. Then, since the right-hand side is finite, we can absorb it into the left-hand side, and we obtain
    \begin{align*}
        \int_0^N \lambda^{p-1} \mathcal{H}^\beta_\infty( \{ x\in \R^d : I_\alpha^\mathcal{D} \mu(x)>\lambda\}) d\lambda \le & 2^{p+1} C \frac{\log (2^{p+1}C)^p}{c^p p} \int_{\R^d} |\mathcal{M}_\alpha \mu|^p \, d\mathcal{H}^\beta_\infty .
    \end{align*} 
     We obtain the desired inequality letting $N\to \infty$, 
    \begin{align*}
        \left( \int_{\R^d} |I_\alpha^\mathcal{D} \mu|^p \, d\mathcal{H}^\beta_\infty \right) ^\frac{1}{p} \le & c_d \frac{\log (2^{p+1}C)}{c } \left( \int_{\R^d} |\mathcal{M}_\alpha \mu|^p \, d\mathcal{H}^\beta_\infty  \right)^\frac{1}{p} \\ 
        \le & \tilde{C} \, p \left( \int_{\R^d} |\mathcal{M}_\alpha \mu|^p \, d\mathcal{H}^\beta_\infty  \right)^\frac{1}{p} 
    \end{align*}
    where $\tilde{C}= \tilde{C}(d, \beta, \alpha)>0$ . 

    Let us consider now the general case. Given a measure $\mu$, we define $\mu_R=\mu\chi_{B(0,R)}$. Since $\mu_R$ has compact support, we have 
    \begin{equation*}
        \Vert I_\alpha^\mathcal{D} \mu_R \Vert_{L^p (\mathcal{H}^\beta_\infty)} \leq C\,  p \Vert \mathcal{M}_\alpha \mu_R \Vert_{L^p(\mathcal{H}^{\beta}_\infty)} \le C\,  p \Vert \mathcal{M}_\alpha \mu \Vert_{L^p(\mathcal{H}^{\beta}_\infty)}.
    \end{equation*}
    Since $\mathcal{H}^{\beta}_\infty$ is continuous from below, we may apply the monotone convergence theorem (see \cite{PS_2023} for instance), and we obtain the desired inequality.

    In order to prove \eqref{MW weak}, let us first assume that $\mu$ has compact support. We may assume that $\Vert \mathcal{M}_\alpha \mu \Vert_{L^{p,\infty}(\mathcal{H}^{\beta}_\infty)}<\infty$, because otherwise there is nothing to prove. Let $N>0$, using again Theorem \ref{good lambda exp}, multiplying by $\lambda^p$ and taking the supremum over $0<\lambda\le N$ we obtain the following estimate,
    \begin{align*}
        \sup_{0<\lambda\le N} \lambda^p \mathcal{H}^\beta_\infty( \{ x\in \R^d : I_\alpha^\mathcal{D} \mu (x)>2\lambda\}) & \le  C e^{-\frac{c}{\varepsilon}} \sup_{0<\lambda \le N} \lambda^p  \mathcal{H}^\beta_\infty(\{ x\in \R^d : I_\alpha^\mathcal{D} \mu (x)>\lambda\}) \\
        & + C \sup_{0<\lambda \le N} \lambda^p  \mathcal{H}^\beta_\infty(\{ x\in \R^d : \mathcal{M}_\alpha \mu (x)>\varepsilon \lambda\}). 
    \end{align*}
    Therefore, 
    \begin{align*}
        \frac{1}{2^p} \sup_{0<\lambda\le 2N} \lambda^p \mathcal{H}^\beta_\infty( \{ x\in \R^d : I_\alpha^\mathcal{D} \mu (x)>\lambda\}) & \le  C e^{-\frac{c}{\varepsilon}} \sup_{0<\lambda \le 2N} \lambda^p  \mathcal{H}^\beta_\infty(\{ x\in \R^d : I_\alpha^\mathcal{D} \mu (x)>\lambda\}) \\
        & + \frac{C}{\varepsilon^p} \Vert \mathcal{M}_\alpha \mu \Vert_{L^{p,\infty}(\mathcal{H}^{\beta}_\infty)}^p.
    \end{align*}

    Observe that the first term is finite because $\mu$ has compact support. Let $B$ be a ball such that $\supp{\mu} \subseteq B$. Using the same argument that before, we have
    \begin{align*}
        \sup_{0<\lambda\le 2N} \lambda^p \mathcal{H}^\beta_\infty( \{ x\in \R^d : I_\alpha \mu (x)>\lambda\}) & \le  \sup_{0<\lambda\le 2N} \lambda^p \mathcal{H}^\beta_\infty( \{ x\in 2B : I_\alpha \mu (x)>\lambda\})\\
        &+ \sup_{0<\lambda\le 2N} \lambda^p \mathcal{H}^\beta_\infty( \{ x\in \R^d : \mathcal{M}_\alpha \mu (x)>\gamma(\alpha) \lambda\})\\
       & \le   (2N)^p r(2B)^\beta +\frac{1}{\gamma(\alpha)^p}\Vert \mathcal{M}_\alpha \mu \Vert_{L^{p,\infty}(\mathcal{H}^{\beta}_\infty)}^p\\
       & <  \infty,
    \end{align*}
    and thus, using \eqref{dyadicequiv1}, we obtain that the dyadic counterpart is also finite.

    We take again $\varepsilon$ such that $2^p C e^{-\frac{c}{\varepsilon}} =\frac{1}{2}$, that is $\varepsilon=\frac{c}{\log (2^{p+1}C)}$. Then, since the right-hand side is finite, we can absorb it into the left-hand side, and we obtain
    \begin{align*}
         \sup_{0<\lambda\le 2N} \lambda \mathcal{H}^\beta_\infty( \{ x\in \R^d : I_\alpha^\mathcal{D} \mu (x)>\lambda\}) ^\frac{1}{p} \le & c_d \frac{\log (2^{p+1}C)}{c }  \Vert \mathcal{M}_\alpha \mu \Vert_{L^{p,\infty}(\mathcal{H}^{\beta}_\infty)}\\
         \le & C' p \Vert \mathcal{M}_\alpha \mu \Vert_{L^{p,\infty}(\mathcal{H}^{\beta}_\infty)}.
    \end{align*}
    Letting $N\to \infty$, we obtain the desired inequality. The case of a general measure $\mu$ is obtained in the same way as before. 
\end{proof}

Using Proposition \ref{PropCruzUribeMoen}, it is easy to see that \eqref{good lambda exp} implies the following good lambda inequality for the Riesz potential.

\begin{corollary}
     Let $0<\alpha <d \in \mathbb{N}$, $\beta\in (d-\alpha, d]$ and $\mu$ be a locally finite measure in $\mathbb{R}^d$. Then, there exist constants $C,c>0$ depending on $d, \alpha, \beta$ such that 
    \begin{align}\label{good lambda exp 2}
    \begin{split}
        \mathcal{H}^\beta_\infty(\{ x\in \R^d :  I_\alpha \mu (x)>2^{d+1}\gamma_1 \lambda, \mathcal{M}_\alpha \mu(x) \le \varepsilon \lambda \}) \le & C e^{-\frac{c}{\varepsilon}}  \mathcal{H}^\beta_\infty(\{ x\in \R^d :  I_\alpha \mu (x)> \gamma_2^{-1} \lambda\}) 
    \end{split}
    \end{align}
    for all $\lambda>0$ and all $0<\varepsilon<1$, where $\gamma_1,\gamma_2>1$ are the constants in \eqref{pointwiseequiv} and \eqref{dyadicequiv1}, respectively.
\end{corollary}

As a consequence of the good lambda inequality with exponential decay \eqref{good lambda exp 2}, we can prove the following result about the exponential integrability of $I_\alpha \mu$  for measures in the local Morrey space $\mathcal{M}^{n-\alpha}(B)$. This is the natural endpoint $p=\infty$ in Theorem \ref{wheedenHausdorffcontent}. We give an alternative proof of a result of Adams and Xiao \cite{AdamsXiao2} (see also \cite[Theorem 1.7]{chen2024selfimproving}).

\begin{theorem}\label{Thm exp int}
     Let $0<\alpha<d \in \mathbb{N}$, $\beta \in( d-\alpha, d ]$ and $\mu$ be a locally finite measure in $\mathbb{R}^d$. Then there exist $C=C(d,\alpha, \beta)>0$ and $\gamma=\gamma(d,\alpha, \beta)>0$ such that 
    \begin{align*}
         \frac{1}{r(2B)^\beta } \int_{2B} \exp \left(  \frac{\gamma \, I_\alpha \mu (x) }{\left\| \mu\right\|_{\mathcal{M}^{d-\alpha}(B)} } \right) \, d \mathcal{H}^\beta_\infty \le C,
    \end{align*}
  for balls $B$ such that $\mu \in \mathcal{M}^{d-\alpha}(B)$. 
\end{theorem}

\begin{proof}
    Let $B$ be a ball in $\R^d$ such that $\left\| \mu\right\|_{\mathcal{M}^{d-\alpha}(B)}<\infty$. We can assume $\left\| \mu\right\|_{\mathcal{M}^{d-\alpha}(B)}\le 1$. Observe that $ r(B)\le |x-y|$ for any $y\in B$ and $x\notin 2 B$, hence
    \begin{equation*}
        I_\alpha \mu (x) = \frac{1}{\gamma(\alpha)} \int_B \frac{d\mu(y)}{|x-y|^{d-\alpha}} \le \frac{1}{\gamma(\alpha)} \frac{\mu(B)}{r(B)^{d-\alpha}} \le \frac{1}{\gamma(\alpha)} \left\| \mu\right\|_{\mathcal{M}^{d-\alpha}(B)} \le \frac{1}{\gamma(\alpha)},  
    \end{equation*}
    for each $x\notin 2 B$.
    
    Therefore, if $\lambda>\frac{1}{\gamma(\alpha)}$ we have $\{x\in \R^d : I_\alpha \mu (x)>\lambda \} \subseteq 2 B$. For fixed $\lambda>\frac{\gamma_2}{\gamma(\alpha)}>\frac{1}{\gamma(\alpha)}$, applying \eqref{good lambda exp 2} with $\varepsilon=\frac{1}{\lambda}$ we obtain 
    \begin{align*}
        \mathcal{H}^\beta_\infty(\{ x\in \R^d : I_\alpha \mu (x)>2^{d+1} \gamma_1 \lambda\}) &\le  C e^{-c\lambda}  \mathcal{H}^\beta_\infty(\{ x\in \R^d : I_\alpha \mu (x)>\gamma_2^{-1}\lambda\}) \\
        & + \mathcal{H}^\beta_\infty(\{ x\in \R^d : \mathcal{M}_\alpha \mu (x)>1\}) \\
       & =  C e^{-c\lambda}  \mathcal{H}^\beta_\infty(\{ x\in \R^d : I_\alpha \mu  (x)>\gamma_2^{-1}\lambda\}) \\
       & \le  C e^{-c\lambda} r(2 B)^\beta .
    \end{align*}

    Using the previous estimate, we can conclude the proof, 
    \begin{align*}
        \int_{ 2 B} \exp \left( \frac{c}{2^{d+2} \gamma_1 } \, I_\alpha \mu (x)  \right) \, d \mathcal{H}^\beta_\infty & =  \frac{c}{2^{d+2} \gamma_1 } \int_0^\infty  e^{\frac{c}{2^{d+2} \gamma_1 }\lambda} \mathcal{H}^\beta_\infty(\{ x\in 2 B : I_\alpha \mu  (x)> \lambda\}) d\lambda \\
       & =  \frac{c}{2} \int_0^\infty  e^{\frac{c}{2}\lambda} \mathcal{H}^\beta_\infty(\{ x\in 2 B : I_\alpha \mu  (x)> 2^{d+1} \gamma_1\lambda\}) d\lambda \\
        & =  \frac{c}{2}   \int_0^{\frac{\gamma_2}{\gamma(\alpha)}}  e^{\frac{c}{2}\lambda} \mathcal{H}^\beta_\infty(\{ x\in 2 B : I_\alpha \mu (x)> 2^{d+1} \gamma_1 \lambda\}) d\lambda \\
        & + \frac{c}{2} \int_{\frac{\gamma_2}{\gamma(\alpha)}}^\infty  e^{\frac{c}{2}\lambda} \mathcal{H}^\beta_\infty(\{ x\in 2 B : I_\alpha \mu (x)> 2^{d+1} \gamma_1 \lambda\}) d\lambda \\
         & \le \frac{c e^{\frac{c}{2}}}{2} \frac{\gamma_2}{\gamma(\alpha)} r(2 B)^\beta     + C \frac{c}{2} r(2 B)^\beta  \int_{\frac{\gamma_2}{\gamma(\alpha)}}^\infty  e^{-\frac{c}{2}\lambda}   d\lambda \\
       & =  C'  r(2 B)^\beta. 
    \end{align*}
\end{proof}

\section{An application to PDEs: exponential integrability of gradient estimates for $p$-Laplace type equation with measure data}

In this section, we use Theorem \ref{Thm exp int} to prove the capacitary local exponential integrability of the gradient of weak solutions of singular $p$-Laplace type equations with measure data. We consider the quasilinear elliptic equation 
\begin{equation}\label{PDE1}
	-\operatorname{div} (A(x,\nabla u ))=\mu,
\end{equation}
in a bounded domain $\Omega\subset \R^d$, where $d\ge 2$. In here, $\mu$ is a finite Radon measure in $\Omega$. The continuous vector field $A = (A_1,...,A_d):\Omega \times \R^d \longrightarrow \R^d$ is $C^1$ in the gradient variable $z$, and it satisfies the following conditions:  
\begin{equation}\label{PDE2}
	\begin{cases}
		|A(x,z)|+|D_z A(x,z)|(|z|^2+s^2)^\frac{1}{2}\le L(|z|^2+s^2)^\frac{p-1}{2},\\
		\nu^{-1} ( |z|^2+ s^2) ^\frac{p-2}{2} |\lambda|^2 \le \langle D_z A(x,z)\lambda , \lambda\rangle,\\
		|A(x,z)-A(x_0, z)| \le L_1 \omega (|x-x_0|) (|z|^2+s^2)^\frac{p-1}{2}
	\end{cases}
\end{equation}
for all $x,x_0\in \Omega$ and $z, \lambda\in \R^d$. The parameters $\nu, L, s, L_1$ are fixed  and such that $0<\nu \le L$ and $s\ge 0$, $L_1\ge 0$, the function $\omega: [0, \infty)\longrightarrow [0,1]$ is a non-decreasing subadditive function satisfying 
\begin{equation}\label{PDE3}
	\lim_{r\downarrow 0} \omega (r) = \omega(0)=0,
\end{equation}
 and the Dini condition
 \begin{equation}\label{PDE4}
 	\int_0^1 \omega(r)\frac{dr}{r}<\infty. 
 \end{equation}
 
 The model example of \eqref{PDE1} is the $p$-Laplacian equation 
 \begin{equation*}
 	-\Delta_p u := - \operatorname{div} (|\nabla u |^{p-2} \nabla u ) = \mu \qquad \text{ in } \Omega.
 \end{equation*}
 We refer to \cite{DM10} for a more detailed explanation of the setting.

 The gradient estimates of the weak solutions of \eqref{PDE1} for $p\ge 2$ have been well studied in the literature \cite{DM10, DM11} in terms of nonlinear potentials. We are interested in the singular case $p\in (1,2)$. In the seminal work \cite{DM10} the authors prove the following gradient estimate for $p\in \left( 2-\frac{1}{d}, 2\right)$,  if $u\in C^1(\Omega)$ solves \eqref{PDE1}, then there exists $C=C(d, p, \lambda, \omega)>0$ such that
 \begin{equation}\label{grad estimate 1}
 	|\nabla u (x)| \le C I_1 (\mu \chi_{B(x,R)}) (x)^\frac{1}{p-1} + C \frac{1}{|B(x,R)|}\int_{B(x,R)} (|\nabla u(y)|+s)dy,
 \end{equation}
 for each ball $B(x,R)\subset \Omega$ with $R\in (0,1]$. Later, the case $p\in \left( \frac{3d-2}{2d-1}, 2-\frac{1}{d}\right)$ was studied in \cite{NP20}, and recently improved in \cite{DZ24}. 
 
 We can prove the following exponential integrability result using Theorem \ref{Thm exp int} and the gradient estimate \eqref{grad estimate 1}. 
 
 \begin{theorem} 
 Let $\beta\in (d-1, d]$, $p\in \left( 2-\frac{1}{d}, 2\right)$ and let $u\in C^1(\Omega)$ is a weak solution of \eqref{PDE1} under the assumptions \eqref{PDE2}--\eqref{PDE4}. Then, there exists $C_1,C_2,C_3>0$ such that 
 			\begin{equation}\label{PDE5}
 				\frac{1}{r(B)^\beta} \int_{B} \exp \left( C_1 \frac{|\nabla u |^{p-1}}{\left\| \mu\right\|_{\mathcal{M}^{d-1}(2B)} } \right) d \mathcal{H}_\infty^\beta \le C_2 \exp \left( \frac{C_3}{|2B|}\int_{2B} \frac{(|\nabla u(y)|+s)}{\left\| \mu\right\|_{\mathcal{M}^{d-1}(2B)} }dy  \right)^{p-1} 
 			\end{equation}
 			for all balls $B$ such that $2B\subset \Omega$ and  $\left\| \mu\right\|_{\mathcal{M}^{d-1}(2B)}<\infty$. 
 \end{theorem} 
 
 \begin{proof}
 	The proof follows the approach of \cite[Theorem 1.5]{HonzikJaye}. 
 	Let $B= B(x_0, R)$ be a ball such that $2B\subset \Omega$. It follows from \eqref{grad estimate 1} that 
 	\begin{equation}
 		|\nabla u (x)|^{p-1} \le C I_1 (\mu \chi_{2B}) (x) + C \left(\frac{1}{|2B|}\int_{2B} (|\nabla u(y)|+s)dy\right) ^{p-1}
 	\end{equation}
 	for all $x\in B$. Applying Theorem \ref{Thm exp int} there exist $C'=C'(d,\beta)>0$ and $\gamma=\gamma(d,\beta)>0$ such that
 	\begin{equation*}
 		\frac{1}{r(2B)^\beta} \int_{2B} \exp \left( \gamma \frac{I_1 \mu(x)}{\left\| \mu\right\|_{\mathcal{M}^{d-1}(2B)}}\right) d\mathcal{H}^\beta_\infty \le C'.
 	\end{equation*}
 	Hence, 
 	\begin{align*}
 		\frac{1}{r(B)^\beta} \int_{B} \exp \left( \frac{\gamma}{C} \frac{|\nabla u(x) |^{p-1}}{\left\| \mu\right\|_{\mathcal{M}^{d-1}(2B)} } \right) d \mathcal{H}_\infty^\beta  \le & C \left(\frac{1}{r(2B)^\beta} \int_{2B} \exp \left( \gamma \frac{I_1 \mu(x)}{\left\| \mu\right\|_{\mathcal{M}^{d-1}(2B)}}\right) d\mathcal{H}^\beta_\infty\right)\\
 		& \cdot \exp \left( \frac{\gamma}{|2B|}\int_{2B} \frac{(|\nabla u(y)|+s)}{\left\| \mu\right\|_{\mathcal{M}^{d-1}(2B)} }dy  \right)^{p-1}  \\
 		\le & C_2 \exp \left(  \frac{\gamma}{|2B|}\int_{2B} \frac{(|\nabla u(y)|+s)}{\left\| \mu\right\|_{\mathcal{M}^{d-1}(2B)} }dy  \right)^{p-1} .
 	\end{align*}
 	This concludes the proof of \eqref{PDE5}.
 \end{proof}

\begin{remark}
    We observe that our techniques cannot reach the endpoint $\beta=d-1$.
\end{remark}

\section*{Acknowledgments}

Y.-W.~B.~Chen is supported by the National Science and Technology Council of Taiwan under research grant number 113-2811-M-002-027. A. Claros is supported by the Basque Government through the BERC 2022-2025 program, by the Ministry of Science and Innovation:  Grant PRE2021-099091 funded by BCAM Severo Ochoa accreditation CEX2021-001142-S/MICIN/AEI/10.13039/501100011033 and by ESF+. 

\section*{Conflicts of Interest}
The authors have no conflicts of interest to declare.

\section*{Data Availability Statement}
Data sharing not applicable to this article as no datasets were generated or analyzed during the current study.

\end{document}